\documentclass[letterpaper,11pt]{amsart}

\usepackage[margin=1.2in]{geometry}
\usepackage{amsmath,amsthm,amssymb}
\usepackage{xspace,xcolor}
\usepackage[breaklinks,colorlinks,citecolor=teal,linkcolor=teal,urlcolor=teal,pagebackref,hyperindex]{hyperref}
\usepackage[alphabetic]{amsrefs}
\usepackage[all]{xy}
\usepackage{color}

\theoremstyle{plain}
\newtheorem{thm}{Theorem}[section]

\newtheorem{lem}[thm]{Lemma}
\newtheorem{prop}[thm]{Proposition}
\newtheorem{cor}[thm]{Corollary}

\theoremstyle{definition}

\newtheorem{eg}[thm]{Example}

\theoremstyle{remark}
\newtheorem{rmk}[thm]{Remark}

\def\Z{{\mathbf Z}}
\def\Q{{\mathbf Q}}

\def\C{{\mathbf C}}

\def\A{{\mathbf A}}

\def\cD{\mathcal{D}}

\def\cH{\mathcal{H}}

\def\cJ{\mathcal{J}}

\def\cM{\mathcal{M}}
\def\cN{\mathcal{N}}
\def\cO{\mathcal{O}}

\def\cT{\mathcal{T}}

\def\frm{\mathfrak{m}}

\def\.{\cdot}
\def\^{\widehat}

\def\({\left(}
\def\){\right)}

\newcommand{\llcurlybracket}{\{\negthinspace\{}
\newcommand{\rrcurlybracket}{\}\negthinspace\}}

\renewcommand{\and}{ \ \ \text{ and } \ \ }

\DeclareMathOperator{\lct} {lct}

\begin{document}

\title{Minimal exponents of hyperplane sections:
a~conjecture of Teissier}

\author{Bradley Dirks}
\author{Mircea Musta\c{t}\u{a}}

\address{Department of Mathematics, University of Michigan, 530 Church Street, Ann Arbor, MI 48109, USA}

\email{bdirks@umich.edu}
\email{mmustata@umich.edu}

\thanks{The authors were partially supported by NSF grant DMS-1701622.}


\begin{abstract}
We prove a conjecture of Teissier asserting that if $f$ has an isolated singularity at $P$ and
$H$ is a smooth hypersurface through $P$, then $\widetilde{\alpha}_P(f)\geq \widetilde{\alpha}_P(f\vert_H)+\frac{1}{\theta_P(f)+1}$,
where $\widetilde{\alpha}_P(f)$ and $\widetilde{\alpha}_P(f\vert_H)$ are the minimal exponents at $P$ of $f$ and $f\vert_H$, respectively,
and $\theta_P(f)$ is an invariant obtained by comparing the integral closures of the powers of the Jacobian ideal of $f$ and of the ideal defining $P$.
The proof builds on the approaches of Loeser \cite{Loeser} and Elduque-Musta\c{t}\u{a} \cite{EM}. The new ingredients are a result concerning the behavior of Hodge 
ideals with respect to finite maps and a result about the behavior of certain Hodge ideals for families of isolated singularities with constant Milnor number. 
In the opposite direction, we show that for every $f$, if $H$ is a general hypersurface through $P$, then
 $\widetilde{\alpha}_P(f)\leq \widetilde{\alpha}_P(f\vert_H)+\frac{1}{{\rm mult}_P(f)}$, extending a result of Loeser from the case of isolated singularities.

\end{abstract}

\maketitle

\section{Introduction}

Let $X$ be a smooth complex algebraic variety and $f\in\cO_X(X)$ nonzero, defining a hypersurface $Y$.
For a point $P\in Y$, the \emph{minimal exponent} $\widetilde{\alpha}_P(f)$ can be defined as the negative of the largest root of the reduced Bernstein-Sato polynomial
of $f$ at $P$. This is a very interesting invariant of singularities that refines the \emph{log canonical threshold} $\lct_P(f)$; more precisely, by a result
of Lichtin and Koll\'{a}r, we have ${\rm lct}_P(f)=\min\{\widetilde{\alpha}_P(f),1\}$ (see \cite[Section~10]{Kollar}). Moreover, by a result of Saito
\cite[Theorem~0.4]{Saito-B} the hypersurface $Y$ has rational singularities at $P$ if and only if $\widetilde{\alpha}_P(f)>1$.
In the setting where $Y$ has an isolated singularity at $P$, the minimal exponent can be described via asymptotic expansion of integrals along vanishing cycles,
see \cite{Malgrange} and \cite{Malgrange2}. In this incarnation,  it has been extensively studied in
\cite{AGZV} and is also known as the \emph{Arnold exponent} of $f$ at $P$.

In this article we are interested in the behavior of the minimal exponent under restriction to a smooth hypersurface $H$ in $X$,
containing $P$. When $Y$ has an isolated singularity at $P$, Teissier introduced and studied in \cite{Teissier2} the invariant $\theta_P(f)$
defined as
\begin{equation}\label{eq_theta}
\theta_P(f)=\max_E\frac{{\rm ord}_E(J_f)}{{\rm ord}_E(\frm_P)},
\end{equation}
where $E$ runs over the prime divisors over $X$ with center at $P$, $\frm_P$ is the ideal defining $P$ in $X$, and $J_f$ is the Jacobian ideal of $f$. 
Using the description of the integral closure of an ideal in terms of divisorial valuations (see \cite[Chapter~9.6.A]{Lazarsfeld}) one can see that the maximum
in (\ref{eq_theta}) is achieved by a divisor on the normalized blow-up of $X$ along $\frm_P\cdot J_f$; moreover, $\theta_P(f)$ is the minimum of the positive rational numbers
$\tfrac{r}{s}$ with the property that $\frm_P^r$ is contained in the integral closure $\overline{J_f^s}$ of $J_f^s$.
The following is our main result, giving a positive answer to a conjecture of Teissier \cite{Teissier3}.

\begin{thm}\label{thm_main}
Suppose that $n=\dim(X)\geq 2$ and the hypersurface defined by $f$ in $X$ has an isolated singularity at $P$. If $H$ is a smooth hypersurface in $X$, with $f\vert_H\neq 0$
and $P\in H$, then
$$\widetilde{\alpha}_P(f)\geq\widetilde{\alpha}_P(f\vert_H)+\frac{1}{\theta_P(f)+1}.$$
\end{thm}

By successively applying the theorem for general hyperplane sections (which automatically have isolated singularities), we obtain the following:

\begin{cor}\label{cor_thm_main}
If  the hypersurface $Y$ defined by $f$ in $X$ has an isolated singularity at $P$ and if $H_1,\ldots,H_{n-1}$ are general smooth hypersurfaces
in $X$, containing $P$, then
$$\widetilde{\alpha}_P(f)\geq\frac{1}{\theta_P(f)+1}+\frac{1}{\theta_P(f\vert_{H_1})+1}+\ldots+\frac{1}{\theta_P(f\vert_{H_1\cap\ldots\cap H_{n-1}})+1}.$$
\end{cor}

The inequality in Theorem~\ref{thm_main} was proved by Loeser \cite{Loeser}, with $\theta_P(f)$ replaced by its round-up $\lceil\theta_P(f)\rceil$
(that is, the smallest integer that is $\geq\theta_P(f)$).
Assuming that $f\in\C[x_1,\ldots,x_n]$, $P=0$, and $H$ is the hyperplane given by $x_n=0$, the argument in \cite{Loeser} made use of the family of hypersurfaces 
$$h_t(x_1,\ldots,x_n)=f(x_1,\ldots,x_{n-1},tx_n)+(1-t)x_n^{\theta+1},$$
where $\theta=\lceil \theta_P(f)\rceil$. 
Note that $\widetilde{\alpha}_0(h_1)=\widetilde{\alpha}_0(f)$ and $\widetilde{\alpha}_0(h_0)=\widetilde{\alpha}_0(f\vert_H)+\frac{1}{\theta+1}$
by the Thom-Sebastiani property of Arnold exponents (see for example \cite[Example~(6.8)]{Malgrange}). The definition of $\theta_0(f)$ implies that 
the Milnor number of $h_t$ at $0$ is constant in a neighborhood $V$ of $1$, hence by a result of Varchenko \cite{Varchenko2} it follows that
$\widetilde{\alpha}_0(h_t)$ is constant on $V$. Finally, by the semicontinuity of the Arnold exponent \cite[Theorem~2.11]{Steenbrink},
it follows that $\widetilde{\alpha}_0(h_t)\geq \widetilde{\alpha}_0(h_0)$ for $t\in V$, and we conclude that 
$\widetilde{\alpha}_0(f)\geq\widetilde{\alpha}_0(f\vert_H)+\frac{1}{\theta+1}$. 

This argument was modified in \cite{EM} to prove the weaker version of Theorem~1 for log canonical thresholds. The idea was to consider
the same family $(h_t)_{t\in\C}$, with $\theta=\theta_P(f)$. In order to make sense of this, one pulls-back this expression by the finite cover
given by $\pi(x_1,\ldots,x_n)=(x_1,\ldots,x_{n-1},x_n^d)$, where $d$ is a divisible enough positive integer. Recall that the log canonical threshold
is characterized by the triviality of certain invariants associated to $f$, the 
\emph{multiplier ideals}, see \cite[Chapter~9]{Lazarsfeld}; due to the presence of the finite cover $\pi$, the argument in 
\cite{EM} relied  on considering whether the equation $x_n^{d-1}$ (defining the 
relative canonical divisor of $\pi$) lies in a suitable multiplier ideal
$h_t\circ\pi$, by making use of various properties of multiplier ideals.

In this note we follow the same approach. Since we deal with the minimal exponent, we need to make use of more refined invariants,
the \emph{Hodge ideals} $I_p(f^{\lambda})$ introduced and studied in \cite{MP1}. 
The definition of these invariants (as well as the proofs of their basic properties) makes use of Saito's theory of mixed Hodge modules \cite{Saito-MHM}. It was shown in \cite{Saito-MLCT} and \cite{MP2} that
in the same way that triviality of multiplier ideals characterizes log canonical thresholds, triviality  of Hodge ideals characterizes minimal exponents. 
In order to extend the approach in \cite{EM} to the setting of Hodge ideals, we need two new properties of these invariants, that are of independent interest.

\begin{thm}\label{thm2}
Let $\pi\colon Y\to X$ be a finite surjective morphism between smooth varieties and let $K_{Y/X}$ be
the effective divisor on $Y$ defined by the determinant of the Jacobian matrix of $\pi$.
If $0\neq f\in\cO_X(X)$ and $g=\pi^*(f)$ both define
reduced divisors, then for every $h\in\cO_X(X)$, every nonnegative integer $p$, and every
 $\lambda\in\Q_{>0}$, if $\pi^*(h)\cdot\cO_Y(-K_{Y/X})\subseteq I_p(g^{\lambda})$, then $h\in I_p(f^{\lambda})$.
\end{thm}

For a more general statement, which does not assume that $f$ and $g$ define reduced divisors, see Theorem~\ref{thm2_v2} below.
We also give a partial converse of this result in the case of a Galois cover (see Theorem~\ref{thm3}). 
At least for such covers, we thus have an extension of the formula in \cite[Theorem~9.5.42]{Lazarsfeld}
relating multiplier ideals under finite maps.

The next result is concerned with certain Hodge ideals associated to families of hypersurfaces with constant Milnor number.
Let $\varphi\colon {\mathcal X}\to T$ be a smooth morphism of smooth complex algebraic varieties (in particular, $T$ is connected), and let $s\colon T\to {\mathcal X}$
be a section of $\varphi$. Let $f\in\cO_{\mathcal X}(\mathcal X)$ be such that $f\circ s=0$ and for every $t\in T$, the restriction $f_t$ to ${\mathcal X}_t=\varphi^{-1}(t)$
is nonzero. We assume that for every $t\in T$, the hypersurface defined by $f_t$ in ${\mathcal X}_t$ is reduced, has at most one singular point at $s(t)$,
and furthermore,
that the Milnor number of this hypersurface at $s(t)$ is independent of $t \in T$ (note that the condition to be reduced is a consequence of isolated singularities
as long as the relative dimension is at least 2). In this case, a result of Varchenko \cite{Varchenko2} says that the spectrum of 
$f_t$ at $s(t)$ is independent of $t\in T$; in particular, the minimal exponent $\widetilde{\alpha}_{s(t)}(f_t)$ is independent of $t\in T$.

\begin{thm}\label{thm4}
With the above notation, if $\alpha=\widetilde{\alpha}_{s(t)}(f_t)$ for $t\in T$, then for every nonnegative integer $p$ and every $\gamma\in\Q\cap (0,1]$
with $p+\gamma \leq\alpha+1$, the subscheme of ${\mathcal X}$ defined by $I_p(f^{\gamma})$ is 
finite and flat over $T$ (possibly empty). Moreover, for every $t\in T$, we have
$$I_p(f_t^{\gamma})=I_p(f^{\gamma})\cdot\cO_{{\mathcal X}_t}.$$
\end{thm}

We deduce this result from the above-mentioned result of Varchenko on the constancy of the spectrum and a result due to Jung, Kim, Yoon, and Saito
\cite{Saito_et_al} that allows us to relate the Hodge ideals satisfying the condition in the theorem with the Hodge filtration in Steenbrink's mixed Hodge structure on the cohomology of the Milnor fiber.
Finally, for the proof of Theorem~\ref{thm_main} we also need the Restriction Theorem for Hodge ideals from \cite{MP1}, as well as 
a version of the Thom-Sebastiani property for certain Hodge ideals.
Regarding the latter property, 
in our setting it is enough to use a Thom-Sebastiani type result
for some related ideals, Saito's \emph{microlocal multiplier ideals}; this property was proved by Maxim, Saito, and Sch\"{u}rmann in \cite{MSS}.

Our main result gives a lower bound for the difference between the minimal exponent of $f$ and 
 the minimal exponent of a restriction of $f$ to a smooth hypersurface, in terms of
Teissier's invariant $\theta_P(f)$. Our last result is an upper bound for the same difference in terms of multiplicity, when we restrict to a
\emph{general} hypersurface. We note that in this result we do not require isolated singularities.

\begin{thm}\label{thm5}
Let $X$ be a smooth complex algebraic variety with ${\rm dim}(X)\geq 2$, $f\in\cO_X(X)$ nonzero, and $P\in X$ such that $f(P)=0$. If $H$ is a general hypersurface in $X$ containing $P$,
then
$$\widetilde{\alpha}_P(f\vert_H)\geq \widetilde{\alpha}_P(f)-\frac{1}{{\rm mult}_P(f)}.$$
\end{thm}

The case when $f$ has an isolated singularity at $P$ is a consequence of a stronger bound proved by Loeser in \cite{Loeser}. 
We deduce the general case from this one by making use of a result from \cite{MP2}. As a consequence of 
Theorem~\ref{thm5} is that we get
conditions, in terms of the minimal exponent $\widetilde{\alpha}_P(f)$, that guarantee that successive general hyperplane sections through $P$
have rational singularities (see Corollary~\ref{cor_thm5}). Another application concerns a characterization of singular points with maximal 
minimal exponent (see Corollary~\ref{cor_max_min_exp}). 

The paper is structured as follows. In the next section we review briefly the Hodge ideals and the microlocal multiplier ideals, the connection between them,
and the corresponding characterization of minimal exponents. In Section 3 we discuss the behavior of Hodge ideals under finite maps and prove
Theorem~\ref{thm2}, as well as its partial converse in the case of Galois finite covers. In Section 4 we recall the relevant result from \cite{Saito_et_al}
and use it to relate certain jumping numbers for Hodge ideals to the spectrum. In particular, we prove Theorem~\ref{thm4}. We combine these results
to give the proof of Theorem~\ref{thm_main} in Section 5. The last section of the paper is devoted to the proof of the bound in the opposite direction in
Theorem~\ref{thm5} above and of the above-mentioned applications.

\subsection{Acknowledgements} We are grateful to Eva Elduque for many discussions related to this project. We would like to thank Sebasti\'{a}n Olano and Jakub Witaszek 
for a conversation that led to Corollary~\ref{cor_max_min_exp}. We are also indebted to the anonymous referees for the careful reading of the paper and for many useful comments.

\section{Hodge ideals and minimal exponents}

In this section we review some basic facts about Hodge ideals and their connection with minimal exponents and microlocal multiplier ideals,
following \cite{MP1} and \cite{MP2}. Let $X$ be a smooth $n$-dimensional complex algebraic variety and $f\in\cO_X(X)$ a nonzero regular function. 
We denote by $\cD_X$ the sheaf of differential operators on $X$.

For every positive rational number $\alpha$, we have a left $\cD_X$-module 
$$\cM(f^{-\alpha}):=\cO_X[1/f]f^{-\alpha}.$$
This is a free module of rank 1 over the sheaf $\cO_X[1/f]$, generated by the element $f^{-\alpha}$, with differential operators acting in the 
expected way: if $D$ is a derivation on $\cO_X$, then
$$D\cdot (hf^{-\alpha})=D(h)f^{-\alpha}-\alpha\frac{hD(f)}{f}f^{-\alpha}.$$
Since $\cM(f^{-\alpha})$ is a filtered direct summand of a mixed Hodge module in the sense of \cite{Saito-MHM}, it carries a canonical filtration
$F_{\bullet}\cM(f^{-\alpha})$,
compatible with the filtration on $\cD_X$ given by order of differential operators. The Hodge ideals $\big(I_p(f^{\alpha})\big)_{p\geq 0}$ describe this filtration. 

In what follows we will only be interested in the case when $f$ defines a reduced divisor $D$. In this case, the Hodge ideals 
are given by 
$$F_p\cM(f^{-\alpha})=I_p(f^{\alpha})\cdot \cO_X(pD)\cdot f^{-\alpha}$$
(we note that $I_p(f^{\alpha})$ was denoted by $I_p(\alpha D)$ in \cite{MP1}). 

It is sometimes convenient to also consider the right $\cD_X$-module corresponding to $\cM(f^{-\alpha})$.
Recall that there is an equivalence of categories between left and right $\cD_X$-modules such that if
$\cM$ is a left $\cD_X$-module, the $\cO_X$-module underlying the right $\cD_X$-module corresponding to $\cM$
is $\omega_X\otimes_{\cO_X}\cM$; see \cite[Section~1.2]{HTT}. We denote by $\cM_r(f^{-\alpha})$ the right $\cD_X$-module corresponding
to $\cM(f^{-\alpha})$. The filtration on $\cM(f^{-\alpha})$ induces a filtration on $\cM_r(f^{-\alpha})$, with the convention
$$F_{p-n}\cM_r(f^{-\alpha})=\omega_X\otimes_{\cO_X}F_p\cM(f^{-\alpha}).$$

We next recall the $V$-filtration associated to $f$. Let $\iota\colon X\to X\times{\mathbf A}^1$ be the graph embedding given by
$\iota(x)=\big(x,f(x)\big)$. We denote the standard coordinate on ${\mathbf A}^1$ by $t$.
The $\cD$-module theoretic push-forward $B_f:=\iota_+\cO_X$ of $\cO_X$ can be described 
as 
$$B_f=\cO_X\otimes_{\C}\C[\partial_t]$$
(see \cite[Example~1.3.5]{HTT}). 
We thus have an $\cO_X$-basis of $B_f$ given by $\partial_t^j\delta$, for $j\geq 0$, where we put $\delta=1\otimes 1\in B_f$.
The action of $t$ on the elements of this basis is given by
$$t\cdot\partial_t^j\delta=f\partial_t^j\delta-j\partial_t^{j-1}\delta.$$

The $V$-filtration on $B_f$ is a rational filtration $(V^{\gamma}B_f)_{\gamma\in\Q}$ which is exhaustive, decreasing, left-continuous and 
discrete\footnote{More precisely, the last two properties mean that there is a positive integer $\ell$ such that $V^{\gamma}B_f$ is constant
for $\gamma\in \big(i/\ell, (i+1)/\ell\big]$, with $i\in\Z$.}.
The filtration, constructed by Malgrange in \cite{Malgrange3}, is uniquely characterized by the following properties:
\begin{enumerate}
\item[i)] Every $V^{\gamma}B_f$ is a coherent module over $\cD_X[t,\partial_tt]$.
\item[ii)] $t\cdot V^{\gamma}B_f\subseteq V^{\gamma+1}B_f$ for all $\gamma\in\Q$, with equality if $\gamma>0$.
\item[iii)] $\partial_t\cdot V^{\gamma}B_f\subseteq V^{\gamma-1}B_f$ for all $\gamma\in\Q$.
\item[iv)] For every $\gamma\in\Q$, the operator $\partial_tt-\gamma$ on ${\rm Gr}_V^{\gamma}:=V^{\gamma}B_f/V^{>\gamma}B_f$ is nilpotent,
where $V^{>\gamma}B_f=\bigcup_{\gamma'>\gamma}V^{\gamma'}B_f$.
\end{enumerate}

\begin{rmk}\label{rmk0_Vfiltration}
It follows from property iv) above that if $\alpha\neq 1$, then $t\partial_t$ is invertible on ${\rm Gr}_V^{\alpha}$; in particular, $\partial_t\colon
{\rm Gr}_V^{\alpha}\to {\rm Gr}_V^{\alpha-1}$ is injective.
This implies that if $u\in B_f$ is such that $\partial_tu\in V^{>0}B_f$, then $u\in V^{>1}B_f$.
\end{rmk}

In particular, the $V$-filtration on $B_f$ induces a $V$-filtration $(V^{\gamma}\cO_X)_{\gamma\in\Q}$ on $\cO_X$ via the inclusion $\cO_X\hookrightarrow B_f$ that maps
$h$ to $h\delta$. Saito introduced in \cite{Saito_microlocal} a \emph{microlocal} version of the $V$-filtration. 
This in turn induces the \emph{microlocal $V$-filtration} $(\widetilde{V}^{\gamma}\cO_X)_{\gamma\in\Q}$ on $\cO_X$, that we can describe via the usual $V$-filtration, as follows. 
For $\gamma\leq 0$, we have $\widetilde{V}^{\gamma}\cO_X=\cO_X$. If $\gamma>0$, write $\gamma=p+\alpha$, for an integer $p$ and
$\alpha\in (0,1]$ (hence $p=\lceil\alpha\rceil-1$). With this notation, $\widetilde{V}^{\gamma}\cO_X$ consists of those regular functions $h\in \cO_X$,
with the property that there are regular functions $h_0,\ldots,h_{p-1}\in\cO_X$ such that
$$h\partial_t^p\delta+h_{p-1}\partial_t^{p-1}\delta+\ldots+h_0\delta\in V^{\alpha}B_f.$$
Whenever the function $f$ is not clear from the context, we write $\widetilde{V}^{\gamma}\cO_X(f)$ for $\widetilde{V}^{\gamma}\cO_X$.

\begin{rmk}\label{rmk_1_V0}
A basic fact is that $\delta\in V^{>0}B_f$. For example, this follows from Sabbah's description of the $V$-filtration in terms of $b$-functions (see \cite{Sabbah})
and the fact, due to Kashiwara \cite{Kashiwara}, that all the roots of the $b$-function of $f$ are negative rational numbers.
\end{rmk}

\begin{rmk}\label{rmk1_microlocal_filtration}
The microlocal $V$-filtration on $\cO_X$ is a rational, decreasing, exhaustive, left-continuous and discrete filtration by coherent ideals. 
With the above definition, the only assertion that is not clear is that $\widetilde{V}^{\gamma_1}\cO_X\subseteq \widetilde{V}^{\gamma_2}\cO_X$
if $\gamma_1>\gamma_2$. In order to check this, we can easily reduce to the case when $\gamma_2=p$ is a positive integer and 
$\gamma_1=p+\alpha$, for a rational number $\alpha\in (0,1]$. In order to prove the inclusion, 
suppose that $h\in \widetilde{V}^{p+\alpha}\cO_X$, hence there are
$h_0,\ldots,h_{p-1}\in\cO_X$ such that
$$h\partial_t^p\delta+h_{p-1}\partial_t^{p-1}\delta+\ldots+h_0\delta\in V^{\alpha}B_f.$$
Since $h_0\delta\in V^{>0}B_f$ by Remark~\ref{rmk_1_V0}, it follows that if 
$$u=h\partial_t^{p-1}\delta+h_{p-1}\partial_t^{p-2}\delta+\ldots+h_1\delta,$$
then $\partial_tu\in V^{>0}B_f$, hence $u\in V^{>1}B_f\subseteq V^1B_f$ by Remark~\ref{rmk0_Vfiltration}. This implies that $h\in \widetilde{V}^p\cO_X$.
\end{rmk}

\begin{rmk}\label{rmk2_microlocal_filtration}
The \emph{microlocal multiplier ideals} of $f$ are defined by
$$\widetilde{\cJ}(f^{\gamma}):=\widetilde{V}^{\gamma+\epsilon}\cO_X\quad\text{for}\quad {0<\epsilon\ll 1}$$
(see \cite{Saito-MLCT} or \cite{MSS}). The shift in the definition is convenient since it implies that for $\gamma<1$, the microlocal multiplier
ideal $\widetilde{\cJ}(f^{\gamma})$ coincides with the usual multiplier ideal $\cJ(f^{\gamma})$ (this is a consequence of a theorem of Budur
and Saito \cite{BS} relating multiplier ideals to the $V$-filtration on $\cO_X$). In what follows we will not shift by $\epsilon$ since,
as we will see shortly, this indexing matches
the one for Hodge ideals, but we will still refer to the elements of the filtration $(\widetilde{V}^{\gamma}\cO_X)_{\gamma\in\Q}$ as \emph{microlocal multiplier ideals}.
\end{rmk}

The following is the main result relating Hodge ideals and microlocal multiplier ideals. It was proved in \cite[Theorem~1]{Saito-MLCT} for 
$\gamma\in\Z_{>0}$ and in \cite[Theorem~A']{MP2} in general.

\begin{thm}\label{Hodge_vs_microlocal}
If $f$ defines a reduced divisor, then for every $\gamma=p+\alpha$, with $p\in\Z_{\geq 0}$ and $\alpha\in\Q\cap (0,1]$, we have
$$I_p(f^{\alpha})+(f)=\widetilde{V}^{\gamma}\cO_X+(f).$$
\end{thm}

In particular, it follows from the theorem that given $P\in X$ with $f(P)=0$, we have 
$I_p(f^{\alpha})=\cO_X$ around $P$ if and only if $\widetilde{V}^{\gamma}\cO_X=\cO_X$ around $P$. 
It was shown by Saito  in \cite{Saito-MLCT} that if $\widetilde{\alpha}_P(f)$ is the \emph{minimal exponent} of $f$ at $P$, then
\begin{equation}\label{eq1_equality_Hodge_microlocal}
\widetilde{V}^{\gamma}\cO_X=\cO_X\,\,\text{around}\,\,P\quad\text{if and only if}\quad \gamma\leq\widetilde{\alpha}_P(f)
\end{equation} (see also
\cite[Remark~6.13]{MP2}). As a consequence, we get the fact that if $f$ defines a reduced divisor, then
\begin{equation}\label{eq2_equality_Hodge_microlocal}
I_p(f^{\alpha})=\cO_X\,\,\text{around}\,\,P\quad\text{if and only if}\quad \widetilde{\alpha}_P(f)\geq p+\alpha
\end{equation} (see 
\cite[Corollary~6.1]{MP2}). 

\begin{rmk}\label{rmk_def_min_exp}
We note that the definition of the minimal exponent that is used in the above results is in terms of the Bernstein-Sato polynomial of $f$.
We will only need the above characterization and thus do not recall the precise definition. For more details and basic properties of the minimal exponent 
that follow from the above characterization in terms of Hodge ideals, we refer to \cite{MP2}. The fact that for isolated singularities the minimal exponent
coincides with the Arnold exponent follows from \cite{Malgrange2}. We will discuss in more detail the case of isolated singularities in Section 4.
\end{rmk}

For us it will be important that Hodge ideals are equal to microlocal multiplier ideals also in an interval of length 1 starting with the minimal exponent.
More precisely, we have the following result. Recall that the \emph{Jacobian ideal} $J_f$ of $f$ is defined as follows: if $x_1,\ldots,x_n$ are algebraic 
coordinates in an open subset $U$ of $X$, then $J_f\vert_U$ is generated by $\frac{\partial f}{\partial x_1},\ldots,\frac{\partial f}{\partial x_n}$
(this definition is independent of the choice of coordinates and thus by gluing the local definitions we get a coherent ideal sheaf of $\cO_X$).

\begin{prop}\label{eq_Hodge_microlocal}
Let $f$ be a nonzero regular function on the smooth complex algebraic variety $X$, defining a reduced divisor, and $P\in X$ such that $f(P)=0$.
Suppose that $\gamma$ is a positive rational number and we write $\gamma=p+\alpha$, with $p=\lceil\gamma\rceil-1$. If
$\gamma\leq\widetilde{\alpha}_P(f)+1$, then
$$I_p(f^{\alpha})=\widetilde{V}^{\gamma}\cO_X$$
in a neighborhood of $P$; moreover, in such a neighborhood these ideals contain $(f)+J_f$.
\end{prop}

\begin{proof}
Suppose first that $\gamma\leq 1$. In this case $p=0$ and both ideals $I_0(f^{\gamma})$ and $\widetilde{V}^{\gamma}\cO_X$ are equal to the multiplier ideal
$\cJ(f^{\gamma-\epsilon})$ for $0<\epsilon\ll 1$ (for the Hodge ideal this follows from \cite[Proposition~9.1]{MP1}, while for the microlocal multiplier
ideal this follows from the result of Budur and Saito \cite{BS} relating multiplier ideals and the $V$-filtration). 
These ideals contain $f$ since 
$(f)=\cJ(f)\subseteq \cJ(f^{\gamma-\epsilon})$ for every $\gamma\leq 1$. 
The fact that $J_f$ is contained in $\cJ(f^{1-\epsilon})$ for $\epsilon>0$ (which, in turn, is contained in $\cJ(f^{\gamma-\epsilon})$) is proved in 
\cite[Theorem~4.2]{ELSV}. 

Suppose now that $\gamma>1$, hence $p\geq 1$, and that $\gamma\leq \widetilde{\alpha}_P(f)+1$. In this case we have
$$I_{p-1}(f^{\alpha})=\cO_X=\widetilde{V}^{\gamma-1}\cO_X$$
in a neighborhood of $P$ by (\ref{eq1_equality_Hodge_microlocal}) and (\ref{eq2_equality_Hodge_microlocal}). 
After possibly replacing $X$ by this neighborhood of $P$, we may and will assume that these equalities hold on $X$.
If we show that $f\in I_p(f^{\alpha})\cap \widetilde{V}^{\gamma}\cO_X$, then 
the equality of the ideals in the proposition follows from Theorem~\ref{Hodge_vs_microlocal}. 
Then the last assertion follows as well if we show that $J_f\subseteq I_p(f^{\alpha})$. 

The fact that $f\in I_p(f^{\alpha})$ has already been noticed in
\cite[Corollary~5.5]{MP2}. The point is that since $I_{p-1}(f^{\alpha})=\cO_X$, we have
$$F_{p-1}\cM(f^{-\alpha})=\left(\cO_X\frac{1}{f^{p-1}}\right)\cdot f^{-\alpha}.$$
The fact that $F_{p-1}\cM(f^{-\alpha})\subseteq F_{p}\cM(f^{-\alpha})$ and the definition of $I_p(f^{\alpha})$ then give
$f\in I_p(f^{\alpha})$. Moreover, since we have $F_1\cD_X\cdot F_{p-1}\cM(f^{-\alpha})\subseteq F_{p}\cM(f^{-\alpha})$,
we see that if $x_1,\ldots,x_n$ are local algebraic coordinates on $X$, then 
$$-(p+\alpha-1)\frac{1}{f^p}\cdot \frac{\partial f}{\partial x_i}f^{-\alpha}=\partial_{x_i}\cdot \left(\frac{1}{f^{p-1}}f^{-\alpha}\right)\in F_p\cM(f^{-\alpha}),$$
hence $J_f\subseteq I_p(f^{\alpha})$.

In order to complete the proof, it is enough to show that $f\in \widetilde{V}^{\gamma}\cO_X$. Since $\widetilde{V}^{\gamma-1}\cO_X=\cO_X$,
it follows that we have in $V^{\alpha}B_f$ an element of the form 
$$u=\partial_t^{p-1}\delta+\text{lower order terms}.$$
In this case we have in $V^{\alpha}B_f$ also the element
$$\partial_ttu=f\partial_t^p\delta+\text{lower order terms},$$
hence $f\in \widetilde{V}^{\gamma}\cO_X$.
\end{proof}

\section{Hodge ideals under finite maps}

In this section we consider the behavior of Hodge ideals under finite surjective morphisms. Let us fix such a morphism
$\pi\colon Y\to X$ between smooth complex $n$-dimensional algebraic varieties. Since we deal with push-forward of $\cD$-modules, in this section
we typically consider right $\cD$-modules. 

Recall that $\pi^*(\cD_X)$ has a canonical structure of $\big(\cD_Y,\pi^{-1}(\cD_X)\big)$-bimodule; as such, it is denoted by $\cD_{Y\to X}$. As for every proper
morphism, we have an induced push-forward morphism between the corresponding derived categories of (quasi-coherent) right $\cD$-modules given by
$R\pi_*(-\otimes^L_{\cD_Y}\cD_{Y\to X})$, see \cite[Chapters~1.3 and 2.5]{HTT}. The case of finite maps is easier: first, $\pi_*$ is exact on quasi-coherent
$\cO_Y$-modules. Second, $\cD_{Y\to X}$ is a flat left $\cD_Y$-module (see \cite[Theorem~2.11.10]{Bjork} or \cite[Proposition~2.10]{FiniteMaps}).
This implies that the functor between derived categories is induced by the exact functor
$\pi_+=\pi_*(-\otimes_{\cD_Y}\cD_{Y\to X})$ between the corresponding Abelian categories of right $\cD$-modules. 
Note that if $\pi$ is \'{e}tale, then $\cD_Y=\pi^*(\cD_X)$ and the functor $\pi_+$ is equal to $\pi_*$ on right $\cD_Y$-modules.

\begin{rmk}\label{rmk_can_morphism_between_pushforwards}
We have a morphism of left $\cD_Y$-modules $\cD_Y\to\cD_{Y\to X}$ that maps $1$ to $\pi^*(1)$. This induces a canonical morphism
of $\cO_X$-modules $\pi_*(\cM)\to\pi_+(\cM)$, which is an isomorphism if $\pi$ is \'{e}tale. 
\end{rmk}

We also write $i_+$ for the push-forward functor between the Abelian categories of $\cD$-modules when $i\colon U\hookrightarrow X$
is the open immersion corresponding to the complement of an effective divisor in $X$. We thus have $i_+\cM=i_*\cM$ for every
right $\cD_U$-module $\cM$. It is straightforward to see (and well-known) that if $j\colon V\hookrightarrow Y$ is the open immersion
with $V=\pi^{-1}(U)$ and $\varphi\colon V\to U$ is the induced morphism,
 then we have a canonical isomorphism of functors $\pi_+\circ j_+\simeq i_+\circ \varphi_+$.

\begin{lem}\label{lem1_finite_maps}
If $\pi\colon Y\to X$ is a finite surjective morphism between smooth complex algebraic varieties and if $\cM$ is a right $\cD_Y$-module having
no torsion as an $\cO_Y$-module, then $\pi_+(\cM)$ has no torsion as an $\cO_X$-module.
\end{lem}

\begin{proof}
The assertion is local on $X$, hence we may assume that $X$ (and thus also $Y$) is affine. By generic smoothness, we can find a nonzero
$h\in\cO_X(X)$ such that $\pi$ is \'{e}tale over the open subset $U=(h\neq 0)$. Let $V=\pi^{-1}(U)$ and $j\colon V\hookrightarrow Y$ and
$i\colon U\hookrightarrow X$ be the corresponding open immersions. Note that we have a canonical morphism 
$\cM\to j_+(\cM\vert_V)$, which is injective since $\cM$ has no torsion as an $\cO_Y$-module. By taking the direct image, we get
an injective morphism
$$\pi_+\cM\to \pi_+j_+(\cM\vert_V)\simeq i_+\varphi_+(\cM\vert_V),$$
where $\varphi\colon V\to U$ is the induced morphism; therefore it is enough to show that the right-hand side has no torsion.
Note that as an $\cO_X$-module, $i_+\varphi_+(\cM\vert_V)$ is simply $i_*\varphi_*(\cM\vert_V)$, since
$\varphi$ is \'{e}tale. 
Since $\cM\vert_V$ is an $\cO_V$-module without torsion, it follows
that $i_*\varphi_*(\cM\vert_V)$ is an $\cO_X$-module without torsion. This completes the proof. 
\end{proof}

In what follows, we will make use of Saito's theory of pure and mixed Hodge modules, for which we refer to \cite{Saito-MHP} and
\cite{Saito-MHM}. Recall that a mixed Hodge module has an underlying filtered (right) $\cD$-module. For example,
on a smooth $n$-dimensional variety $X$, we have the pure Hodge module $\Q_X^H[n]$, whose underlying $\cD_X$-module
is $\omega_X$ (the right $\cD_X$-module corresponding to $\cO_X$), with the filtration given by $F_{p-n}\omega_X=\omega_X$ for
$p\geq 0$ and $F_{p-n}\omega_X=0$, otherwise. 

Suppose now that $\varphi\colon W\to X$ is a morphism between smooth varieties which is either finite and surjective or
an open immersion, given by the complement of an effective divisor. If 
$(\cM,F_{\bullet})$ is a filtered $\cD_W$-module on $W$ that underlies a 
mixed Hodge module $M$, then we have by \cite{Saito-MHM} a mixed Hodge module that we will denote by $\varphi_+M$
and whose underlying filtered $\cD_X$-module we will denote by $\varphi_+(\cM,F_{\bullet})$. The corresponding $\cD_X$-module
is just $\varphi_+(\cM)$, but the description of the filtration is rather subtle. 
When $W$ is the complement of the hypersurface defined by $f\in\cO_X(X)$ and $M={\mathbf Q}_W^H[n]$, we get the filtration on $\cO_X[1/f]$ described by the Hodge ideals
of $f$. 
One easy case is that when $\varphi$ is finite and \'{e}tale, in which case
$F_p(\varphi_+\cM)=\varphi_*(F_p\cM)$ for all $p$. 

Given a finite surjective morphism $\pi\colon Y\to X$ of smooth $n$-dimensional algebraic varieties, we have a canonical
morphism of mixed Hodge modules
\begin{equation}\label{eq_can_map_Q}
\Q_X^H[n]\to\pi_+\Q_Y[n]
\end{equation}
that commutes with restriction to open subsets of $X$. 
At the level of $\cD_X$-modules, this is given by the composition
$$\omega_X\hookrightarrow \pi_*\omega_Y\to\pi_+\omega_Y,$$
where the first morphism maps a form $\eta\in\omega_X$ to its pull-back $\pi^*(\eta)\in\omega_Y$
and the second morphism is the one in Remark~\ref{rmk_can_morphism_between_pushforwards}.


The next lemma provides the ingredient to relate Hodge ideals under finite maps. Suppose that $\pi\colon Y\to X$
is a finite surjective morphism between $n$-dimensional smooth varieties, $f\in\cO_X(X)$ is nonzero, and $g=f\circ\pi\in\cO_Y(Y)$.
Recall that associated to $f$ and $g$ we have the filtered right $\cD$-modules $\cM_r(f^{-\alpha})$ and $\cM_r(g^{-\alpha})$ 
on $X$ and $Y$, respectively.
Note that for every $\alpha\in\Q_{>0}$, we have a canonical morphism of $\cO_X$-modules
\begin{equation}\label{eq0_morphism}
\cM_r(f^{-\alpha})\to \pi_*\cM_r(g^{-\alpha})
\end{equation}
that maps $u=f^{-\alpha}\frac{\eta}{f^m}$, with $\eta\in\omega_Y$, to $\pi^*(u):=g^{-\alpha}\frac{\pi^*(\eta)}{g^m}$.

\begin{lem}\label{lem2_finite_maps}
With the above notation, for every $\alpha\in\Q_{>0}$, the map $\tau$ given by the composition
$$\cM_r(f^{-\alpha})\to \pi_*\cM_r(g^{-\alpha})\to\pi_+\cM_r(g^{-\alpha}),$$
where the first map is the one in (\ref{eq0_morphism}) and the second map is the one in Remark~\ref{rmk_can_morphism_between_pushforwards},
is an injective strict morphism of filtered $\cD_X$-modules. 
\end{lem}

\begin{proof}
For the proof, we will need to make use of the definition of the filtrations on $\cM_r(f^{-\alpha})$ and $\cM_r(g^{-\alpha})$
in \cite[Section~2]{MP1}. Let $X'=X\smallsetminus V(f)$ and $Y'=Y\smallsetminus V(g)$. Choose an integer $m\geq 2$ such that $m\alpha \in \Z$
and let 
$$Y''={\mathcal Spec}\big(\cO_{Y'}[y]/(y^m-g^{-m\alpha})\big)\quad\text{and}\quad X''={\mathcal Spec}\big(\cO_{X'}[y]/(y^m-f^{-m\alpha})\big),$$
so that we have a diagram with Cartesian squares
$$
\xymatrix{
Y''\ar[r]^q \ar[d]_{\psi} & Y'\ar[d]_{\varphi}\ar[r]^{j} & Y\ar[d]_{\pi} \\
X''\ar[r]^p & X'\ar[r]^i & X,
}
$$
with $i$ and $j$ open immersions and $p$ and $q$ finite \'{e}tale morphisms.
We have by (\ref{eq_can_map_Q}) a canonical morphism of mixed Hodge modules
$$\Q_{X''}^H[n]\to\psi_+\Q_{Y''}^H[n],$$
that induces after applying $i_+p_+$ and taking the underlying filtered $\cD_X$-modules,
a morphism of filtered $\cD_X$-modules
$$\bigoplus_{k=0}^{m-1}\cM_r(f^{-k\alpha})\to \pi_+\bigoplus_{k=0}^{m-1}\cM_r(g^{-k\alpha}).$$
By taking the suitable eigenspace with respect to the $\Z/m\Z$ action on both sides, we obtain, for $k=1$,
a morphism of filtered $\cD_X$-modules 
$$\cM_r(f^{-\alpha})\to\pi_+\cM_r(g^{-\alpha}).$$
Checking that it is given by the composition in the statement is an easy exercise. Strictness follows from the fact that for every morphism
of mixed Hodge modules, the underlying morphism of filtered $\cD$-modules is strict. 
The fact that $\tau$ is injective is clear if $\pi$ is \'{e}tale; the general case follows by restricting to an open subset
$U$ of $X$ such that $\pi$ is \'{e}tale over $U$ and using the fact that as an $\cO_X$-module, $\cM_r(f^{-\alpha})$ has no torsion.
\end{proof}

In order to describe the filtration on $\pi_+\cM_r(g^{-\alpha})$, we take the usual approach, by factoring $\pi$ as
$q\circ \rho$, where $\rho\colon Y\to Y\times X$ is the graph embedding given by $\rho(y)=\big(y,\pi(y)\big)$ and
$q\colon Y\times X\to X$ is the projection onto the second component. Let us assume that we have algebraic coordinates 
$x_1,\ldots,x_n$ defined on $X$ (we can always reduce to this case by taking a suitable cover of $X$). Let $\pi_i=\pi^*(x_i)\in\cO_Y(Y)$.
If $\cM$ is a right $\cD_Y$-module, then the $\cD$-module push-forward via $\rho$ is easy to compute: we have an isomorphism
\begin{equation}\label{eq_isom_immersion}
\rho_+\cM\simeq \cM\otimes_{\C}\C[\partial_{x_1},\ldots,\partial_{x_n}],
\end{equation}
where on the right-hand side a function $h\in\cO_X$ acts via
$$(u\otimes 1)h=u\pi^*(h)\otimes 1\quad\text{for all}\quad u\in\cM$$
and a derivation $D\in {\rm Der}_{\C}(\cO_Y)$ acts by
\begin{equation}\label{eq_action_od_derivations}
(u\otimes \partial_x^{\beta})D=uD\otimes\partial_x^{\beta}-\sum_{i=1}^nu D(\pi_i)\otimes\partial_x^{\beta+e_i}\quad\text{for}\quad u\in\cM, 
\beta\in\Z_{\geq 0}^n,
\end{equation}
where we use the multi-index notation and $e_1,\ldots,e_n$ is the standard basis of $\Z^n$. 
Moreover, if $(\cM,F_{\bullet})$ is a filtered $\cD_Y$-module, then via the isomorphism (\ref{eq_isom_immersion})
we have
\begin{equation}\label{filtration_rho}
F_k\rho_+\cM=\sum_{\beta\in\Z_{\geq 0}^n}F_{k-|\beta|}\cM\otimes \partial_x^{\beta}\quad\text{for all}\quad k\in\Z,
\end{equation}
where for $\beta=(\beta_1,\ldots,\beta_n)$, we put $|\beta|=\sum_i\beta_i$. 

On the other hand, the $\cD$-module push-forward via $q$ is computed by the relative Spencer complex. Given a right $\cD_{Y\times X}$-module
$\cN$, the relative Spencer complex of $\cN$ is the complex
$$C^{\bullet}(\cN):\,\,0\longrightarrow\cN\otimes_{\cO_{Y\times X}}\wedge^np^*(\cT_Y)\longrightarrow \ldots\longrightarrow \cN\otimes_{\cO_{Y\times X}}p^*(\cT_Y)
\overset{d_1}\longrightarrow\cN\to 0,$$
where $\cT_Y$ is the tangent sheaf of $Y$ and $p\colon Y\times X\to Y$ is the projection onto the first component. 
The map $d_1$, which is the only one we will need, is given by right multiplication. 

If $\cM$ is a $\cD_Y$-module, then we have a canonical isomorphism
$$\pi_+\cM\simeq \cH^0\big(C^{\bullet}(\rho_+\cM)\big)={\rm coker}\big(\rho_+\cM\otimes p^*(\cT_Y)\to \rho_+\cM\big).$$
Moreover, if $(\cM,F_{\bullet})$ is the filtered $\cD_Y$-module underlying a mixed Hodge module, then via this isomorphism,
the filtration on
$\pi_+\cM$ is the quotient filtration induced by the filtration on $\rho_+\cM$ described in (\ref{filtration_rho}).

\begin{rmk}\label{rmk_description_tau}
If we are in the setting of Lemma~\ref{lem2_finite_maps}, then via the 
isomorphism 
\begin{equation}\label{eq_rmk_description_tau}
\pi_+\cM_r(g^{-\alpha})\simeq {\rm coker}\big(\rho_+\cM_r(g^{-\alpha}) \otimes p^*(\cT_Y)\to \rho_+\cM_r(g^{-\alpha})\big),
\end{equation}
the morphism $\tau$ maps $u\in \cM_r(f^{-\alpha})$ to the class of $\pi^*(u)\otimes 1\in \rho_+\cM_r(g^{-\alpha})$. Indeed, since 
as an $\cO_X$-module $\pi_+\cM_r(g^{-\alpha})$ has no torsion by Lemma~\ref{lem1_finite_maps}, in order to check the assertion
we may restrict to on open subset $U$ of $X$ such that $\pi$ is \'{e}tale over $U$; in this case the assertion follows via an easy computation.
\end{rmk}

\begin{thm}\label{thm2_v2}
Let $\pi\colon Y\to X$ be a finite surjective morphism between smooth $n$-dimensional complex algebraic varieties. If $f\in\cO_X(X)$ is nonzero
and $g=\pi^*(f)\in\cO_Y(Y)$, then for every $\alpha\in\Q_{>0}$ and $k\in\Z$, we have the following inclusion
$$\{u\in \cM_r(f^{-\alpha})\mid \pi^*(u)\in F_k\cM_r(g^{-\alpha})\}\subseteq F_k\cM_r(f^{-\alpha}).$$
\end{thm}

Note that if we assume that $f$ and $g$ define reduced divisors, then by passing from right to left $\cD$-modules and using the definition
of Hodge ideals, we obtain the assertion in Theorem~\ref{thm2}.

\begin{proof}[Proof of Theorem~\ref{thm2_v2}]
After taking a suitable open cover of $X$, we may assume that we have algebraic coordinates defined on $X$. 
If $u\in \cM_r(f^{-\alpha})$ is such that $\pi^*(u)\in F_k\cM_r(g^{-\alpha})$, then it follows from the above description of the filtration 
on $\pi_+\cM_r(g^{-\alpha})$ that the class of $\pi^*(u)\otimes 1$ in ${\rm coker}\big(\rho_+\cM_r(g^{-\alpha}) \otimes p^*(\cT_Y)\to \rho_+\cM_r(g^{-\alpha})\big)$
lies in $F_k\big(\pi_+\cM_r(g^{-\alpha})\big)$. Since this is equal to $\tau(u)$ and $\tau$ is strict by Lemma~\ref{lem2_finite_maps}, we conclude that
$u\in F_k\cM_r(f^{-\alpha})$. 
\end{proof}

Our next goal is to understand how far the inclusion in Theorem~\ref{thm2_v2} is from being an equality. 
We will do this under the extra assumption that $\pi$ is Galois. 
This result will not be used in the following sections.

Let $\pi\colon Y\to X$ be a finite surjective morphism between smooth complex algebraic varieties. Recall that $\pi$ is Galois if the corresponding
field extension $\C(X)\hookrightarrow\C(Y)$ is normal (hence Galois, since we are in characteristic $0$). Suppose that this is the case and let 
$G$ be the Galois group of this field extension. Since $Y$ is the integral closure of $X$ in $\C(Y)$, we have an induced action of $G$ on $Y$ and 
$\pi$ is the quotient morphism with respect to this action. 

\begin{thm}\label{thm3}
Let $\pi\colon Y\to X$ be a Galois finite surjective morphism between smooth $n$-dimensional complex algebraic varieties. 
Suppose that $f\in\cO_X(X)$ is nonzero
and let $g=\pi^*(f)\in\cO_Y(Y)$. For every $\alpha\in\Q_{>0}$ and $k\in\Z$, we put
$$F'_k\cM_r(f^{-\alpha}):=\{u\in \cM_r(f^{-\alpha})\mid \pi^*(u)\in F_{k}\cM_r(g^{-\alpha})\}.$$
We then have for every $k\in\Z$
\begin{equation}\label{eq_thm3}
F_k\cM_r(f^{-\alpha})=\sum_{i\geq 0}F'_{k-i}\cM_r(f^{-\alpha})\cdot F_i\cD_X.
\end{equation}
In particular, we have 
$$F_k\cM_r(f^{-\alpha})=F_{k-1}\cM_r(f^{-\alpha})\cdot F_1\cD_X+F'_k\cM_r(f^{-\alpha}).$$
\end{thm}

\begin{proof}
After taking a suitable open cover of $X$, we may and will assume that $X$ (hence also $Y$) is affine and we have algebraic coordinates 
$x_1,\ldots,x_n$ defined on $X$. We may thus use the 
description of the filtration on $\pi_+\cM_r(g^{-\alpha})$ that we have discussed. 

Note that the action of $G$ on $Y$ induces a $G$-action on $\cO_Y(Y)$. Moreover, it 
also induces a $G$-action on $\Gamma(Y,\omega_Y)$
that makes $\omega_Y$ a $G$-equivariant sheaf\footnote{Since $G$ is a finite group and $Y$ is affine, this simply means that the scalar multiplication
$\cO_Y(Y)\times \Gamma(Y,\omega_Y) \to \Gamma(Y,\omega_Y)$ is compatible with the $G$-actions.}.
  Moreover, its subspace of invariant sections is precisely the image of the map
  $$\Gamma(X,\omega_X)\to \Gamma(Y,\omega_Y)$$
  given by the pull-back of differential forms (see \cite[Theorem~1]{Brion}). Since $g$ is a $G$-invariant section of $\cO_Y$, it follows that the sheaf 
  $\cM_r(g^{-\alpha})$
  has an induced structure of $G$-equivariant sheaf; moreover, its subspace of $G$-invariant sections is the image of the pull-back map
  $$\Gamma\big(X,\cM_r(f^{-\alpha})\big)\to \Gamma(Y,\cM_r(g^{-\alpha})\big).$$
  
  We also have an induced $G$-action on $\Gamma(X,\pi_+\cM_r(g^{-\alpha})\big)$. In order to see this, note that we have a natural
  $G$-action on $\Gamma\big(Y\times X,\rho_+\cM_r(g^{-\alpha})\big)=\Gamma\big(Y,\cM_r(g^{-\alpha})\big)\otimes_{\C}\C[\partial_{x_1},\ldots,
  \partial_{x_n}]$, where $G$ acts trivially on $\C[\partial_{x_1},\ldots,\partial_{x_n}]$. The $G$-action on $Y$ induces an action of $G$ on 
  $\Gamma(Y,\cT_Y)$ such that the multiplication map $\rho_+\cM_r(g^{-\alpha})\otimes\cT_Y\to \rho_+\cM_r(g^{-\alpha})$ is compatible with 
  the $G$-actions.
  We thus get an induced $G$-action on the cokernel of this map and thus on $\Gamma(X,\pi_+\cM_r(g^{-\alpha})\big)$ via the isomorphism 
  (\ref{eq_rmk_description_tau}). It also follows from the description of
  the canonical morphism
  $$\tau\colon \Gamma\big(X,\cM_r(f^{-\alpha})\big)\to\Gamma\big(X,\pi_+\cM_r(g^{-\alpha})\big)$$
  in Remark~\ref{rmk_description_tau} that the image of $\tau$ lands in the subspace of $G$-invariant sections. 
  A key point is that due to the fact that $g$ is $G$-invariant, the filtration on $\Gamma\big(Y,\cM_r(g^{-\alpha})\big)$ is preserved by the $G$-action
  and therefore so is the filtration on $\Gamma\big(X,\pi_+\cM_r(g^{-\alpha})\big)$. 
  
  Given a complex vector space $V$ with a $G$-action, we consider the linear map $A=A_V\colon V\to V$ given by
  $A(v)=\frac{1}{|G|}\sum_{g\in G}gv$. Note that $A(v)$ lies in the subspace $V^G$ of invariant elements for every $v\in V$ and $A(v)=v$ if $v\in V^G$. 
  
The inclusion ``$\supseteq$" in (\ref{eq_thm3}) is clear: it follows from Theorem~\ref{thm2_v2} that for every $j$ we have 
$F'_j\cM_r(f^{-\alpha})\subseteq F_j\cM_r(f^{-\alpha})$ and thus
$$F'_{k-i}\cM_r(f^{-\alpha})\cdot F_i\cD_X\subseteq F_{k}\cM_r(f^{-\alpha}).$$
We now prove the reverse inclusion. Given a global section $u$ of $F_{k}\cM_r(f^{-\alpha})$, it follows from Lemma~\ref{lem2_finite_maps}
that $\tau(u)$ is a global section of $F_k\big(\pi_+\cM_r(g^{-\alpha})\big)$. This means that we can find global sections
$w_{\beta}$ of $F_{k-|\beta|}\cM_r(g^{-\alpha})$ for
$\beta\in\Z_{\geq 0}^n$, with $w_{\beta}=0$ for all but finitely many $\beta$, such that 
$$\pi^*(u)\otimes 1-\sum_{\beta}w_{\beta}\otimes \partial_x^{\beta}\in {\rm Im}\big(\rho_+\cM_r(g^{-\alpha})\otimes\cT_Y\to \rho_+\cM_r(g^{-\alpha})\big).$$
If we put $w'_{\beta}=A(w_{\beta})\in \Gamma(Y, F_{k-|\beta|}\cM_r(g^{-\alpha})\big)^G$, we see that
$$\pi^*(u)\otimes 1-\sum_{\beta}w'_{\beta}\otimes \partial_x^{\beta}\in {\rm Im}\big(\rho_+\cM_r(g^{-\alpha})\otimes\cT_Y\to \rho_+\cM_r(g^{-\alpha})\big).$$
Since each $w'_{\beta}$ is $G$-invariant, it follows that we can write $w'_{\beta}=\pi^*(u_{\beta})$ for some
global section $u_{\beta}$ of $F'_{k-|\beta|}\cM_r(f^{-\alpha})$. 
We now claim that for every $\beta$, $\pi^*(u_{\beta})\otimes \partial_x^{\beta}$ and $\pi^*(u_{\beta}\partial_x^{\beta})\otimes 1$
have the same image in $\Gamma\big(X, \pi_+\cM_r(g^{-\alpha})\big)$. If this is the case, it follows that 
$\tau(u)=\tau\left(\sum_{\beta}u_{\beta}\partial_x^{\beta}\right)$.
Since $\tau$ is injective by Lemma~\ref{lem2_finite_maps}, we conclude that $u\in\sum_{i\geq 0}F'_{k-i}\cM_r(f^{-\alpha})\cdot F_i\cD_X$. 

In order to prove the above claim, since $\pi_+\cM_r(g^{-\alpha})$ is an $\cO_X$-module with no torsion by Lemma~\ref{lem1_finite_maps},
it is enough to prove it on some nonempty open subset of $X$. We thus may and will assume that $\pi$ is \'{e}tale. In this case, if $\pi_i=x_i\circ\pi$,
then $\pi_1,\ldots,\pi_n$ give an algebraic system of coordinates on $Y$ and we have a corresponding system of derivations 
$\partial_{\pi_1},\ldots,\partial_{\pi_n}$. 
Of course, arguing by induction on
$|\beta|$, it is enough to show that for every global section $\eta$ of $\cM_r(f^{-\alpha})$, the elements 
$\pi^*(\eta\partial_{x_i})\otimes 1$ and $\pi^*(\eta)\otimes\partial_{x_i}$ in $\rho_+\cM_r(g^{-\alpha})$ 
have the same image in $\pi_+\cM_r(g^{-\alpha})$.
This follows from the fact that by (\ref{eq_action_od_derivations}), the map $\rho_+\cM_r(g^{-\alpha})\otimes\cT_Y\to \rho_+\cM_r(g^{-\alpha})$ maps
$\pi^*(\eta)\otimes\partial_{\pi_i}$ to 
$$\pi^*(\eta)\partial_{\pi_i}\otimes 1-\pi^*(\eta)\otimes \partial_{x_i}=\pi^*(\eta\partial_{x_i})\otimes 1-\pi^*(\eta)\otimes\partial_{x_i}.$$
This completes the proof of (\ref{eq_thm3}) and the last assertion in the theorem is an immediate consequence. 
\end{proof}

We now give a consequence of the results in Theorems~\ref{thm2} and \ref{thm3} to minimal exponents.

\begin{cor}\label{cor_finite_maps}
Let $\pi\colon Y\to X$ be a finite surjective morphism between smooth complex algebraic varieties and let $K_{Y/X}$ be the effective divisor
on $Y$ locally defined by the determinant of a Jacobian matrix of $\pi$. If $0\neq f\in\cO_X(X)$ and $g=f\circ \pi$ both define reduced divisors, then the following hold
for every $P$ in $X$ with $f(P)=0$, every nonnegative integer $p$, and every $\alpha\in\Q\cap (0,1]$:
\begin{enumerate}
\item[i)] If $\cO_Y(-K_{Y/X})\subseteq I_p(g^{\alpha})$ in a neighborhood of the fiber $\pi^{-1}(P)$, then $\widetilde{\alpha}_P(f)\geq p+\alpha$.
\item[ii)] If $\pi$ is Galois and the hypersurface defined by $f$ is not smooth at $P$, then the converse of i) holds: if $\widetilde{\alpha}_P(f)\geq p+\alpha$, then
$\cO_Y(-K_{Y/X})\subseteq I_p(g^{\alpha})$ in a neighborhood of the fiber $\pi^{-1}(P)$.
\end{enumerate}
\end{cor}

\begin{rmk}
Note that a similar assertion holds for log canonical thresholds of arbitrary regular functions in the setting of a finite surjective morphism between smooth varieties, as above:
we have ${\rm lct}_P(f)>\lambda$ if and only if $\cO_Y(-K_{Y/X})$ is contained in the multiplier ideal $\cJ(g^{\lambda})$ in a neighborhood of the fiber
$\pi^{-1}(P)$. This follows from the fact that ${\rm lct}_P(f)>\lambda$ if and only if $\cJ(f^{\lambda})=\cO_X$ in a neighborhood of $P$ and the theorem relating
the multiplier ideals of $f$ and $g$ (see \cite[Theorem~9.5.42]{Lazarsfeld}).
\end{rmk}

\begin{proof}[Proof of Corollary~\ref{cor_finite_maps}]
The first assertion follows directly from Theorem~\ref{thm2} and the characterization of $\widetilde{\alpha}_P(f)$ via the Hodge ideals of $f$ in (\ref{eq2_equality_Hodge_microlocal}). Suppose now that we are in the setting of ii). Using again the characterization of $\widetilde{\alpha}_P(f)$ 
via the Hodge ideals of $f$, we conclude using the hypothesis that
$I_p(f^{\alpha})=\cO_X$ in a neighborhood of $P$. Equivalently, we have
\begin{equation}\label{eq1_cor_finite_maps}
\cO_X\cdot \frac{1}{f^p}f^{-\alpha}=F_p\cM(f^{-\alpha}).
\end{equation}
For every $k\in\Z$, let us put 
$$F'_k\cM(f^{-\alpha})=\left\{u=\frac{h}{f^p}f^{-\alpha}\in F_k\cM(f^{-\alpha})\mid \frac{\pi^*(h)}{g^p}\cO_Y(-K_{Y/X})\cdot g^{-\alpha}
\subseteq F_k\cM(g^{-\alpha})\right\}.$$
It follows from equation (\ref{eq_thm3}) in
Theorem~\ref{thm3} (after passing from right to left $\cD$-modules) that 
$$F_p\cM(f^{-\alpha})=F'_p\cM(f^{-\alpha})+F_1\cD_X\cdot F_{p-1}\cM(f^{-\alpha}).$$

We next show that since $f$ has a singular point at $P$, we have
\begin{equation}\label{eq2_cor_finite_maps}
F_1\cD_X\cdot F_{p-1}\cM(f^{-\alpha})\subseteq\frm_P\cdot \frac{1}{f^p}f^{-\alpha},
\end{equation}
where $\frm_P$ is the ideal of functions vanishing at $P$. In order to see this, let us choose local coordinates $x_1,\ldots,x_n$ centered at $P$.
First, recall that 
$$F_{p-1}\cM(f^{-\alpha})\subseteq \cO_X\cdot \frac{1}{f^{p-1}}f^{-\alpha}\subseteq \frm_P\cdot\frac{1}{f^p}f^{-\alpha},$$
where the second inclusion follows from the fact that $f\in\frm_P$. Moreover, for every $i$ and every $h\in\cO_X$, we have
$$\partial_{x_i}\cdot \frac{h}{f^{p-1}}f^{-\alpha}=\frac{\partial h}{\partial x_i}\cdot\frac{f}{f^p}f^{-\alpha}-
(\alpha+p-1)\frac{\partial f}{\partial x_i}\cdot\frac{h}{f^p}f^{-\alpha}\in  \frm_P\cdot \frac{1}{f^p}f^{-\alpha},$$
where we use the fact that $f,\frac{\partial f}{\partial x_i}\in\frm_P$. This proves (\ref{eq2_cor_finite_maps}).

We thus have the $\cO_X$-submodule $F'_p\cM(f^{-\alpha})$ of $\cO_X\cdot \frac{1}{f^p}f^{-\alpha}$ and
(\ref{eq1_cor_finite_maps}) and (\ref{eq2_cor_finite_maps}) give
$$\cO_X\cdot \frac{1}{f^p}f^{-\alpha}\subseteq F'_p\cM(f^{-\alpha})+\frm_P\cdot \frac{1}{f^p}f^{-\alpha}.$$
We deduce using Nakayama's lemma that $\frac{1}{f^p}f^{-\alpha}\in F'_p\cM(f^{-\alpha})$ around $P$, that is,
$\cO_Y(-K_{Y/X})\subseteq I_p(g^{\alpha})$. 
\end{proof}

\section{Hodge ideals for families with constant Milnor number}

Let $X$ be a smooth $n$-dimensional complex algebraic variety and $f\in\cO_X(X)$ a nonzero regular function.
Let $P\in X$ be a point with $f(P)=0$; we assume that $P$ is a singular point of $f$ and that $f$ has an isolated singularity at $P$.
Recall that $J_f$ denotes the Jacobian ideal of $f$, generated by $\frac{\partial f}{\partial x_1},\ldots,\frac{\partial f}{\partial x_n}$,
where $x_1,\ldots,x_n$ are local coordinates on $X$. Our assumption on $f$ implies that the dimension $\mu=\mu_P(f):=\dim_{\C}(\cO_{X,P}/J_{f,P})$ is a finite
positive number; this is the \emph{Milnor} number of $f$ at $P$. We note that a more natural context for the discussion in this section is that
of holomorphic functions on complex manifolds; however, this is not really that different from the algebraic case since we deal with isolated singularities,
so that we prefer to stick to the algebraic case, as in the rest of the article. 

For basic facts on the spectrum of $f$ and its connection to the $V$-filtration and the mixed Hodge structure on the cohomology of the Milnor fiber of 
$f$ at $P$, we refer to \cite[Section~1]{Saito_et_al} and the references therein. We only recall that if $F_{f,P}$ denotes the Milnor fiber of $f$ at $P$,
then on its cohomology $H^{n-1}(F_{f,P},\Q)$ there is a mixed Hodge structure and a compatible action of the monodromy $T$ (the \emph{inverse} of the Milnor monodromy). Note that $\dim_{\Q}H^{n-1}(F_{f,P},\Q)=\mu$.
 If $T_s$ is the semisimple
part of the monodromy and if $\lambda$ is an eigenvalue of $T_s$ on the above cohomology, then $H^{n-1}(F_{f,P},\C)_{\lambda}$ denotes the corresponding
eigenspace.

The \emph{spectrum} of $f$ at $P$ is a collection of $\mu$ positive rational numbers, not necessarily distinct and indexed non-decreasingly
$\alpha_1,\ldots,\alpha_{\mu}$ such that for every $\beta\in\Q_{>0}$, we have
\begin{equation}\label{eq1_spectrum}
\#\{i\mid\alpha_i=\beta\}=\dim_{\C}{\rm Gr}^q_FH^{n-1}(F_{f,P},\C)_{\lambda},
\end{equation}
where $\lambda={\rm exp}(-2\pi i\beta)$ and $q=n-\lceil \beta\rceil$. Note that $\alpha_1$ appears with multiplicity 1 and it is equal to
$\widetilde{\alpha}_P(f)$. 

Let us review now the connection between the Hodge filtration on $H^{n-1}(F_{f,P},\Q)$ and the $V$-filtration on the Brieskorn lattice. Recall that 
the \emph{Brieskorn lattice} of $f$ at $P$ is
$$H_{f,P}'':=\Omega_{X,P}^n/df\wedge d\Omega_{X,P}^{n-2},$$
where the stalks of the sheaves of differential forms are the analytic ones (associated to the complex manifold $X^{\rm an}$). This has a structure of
free module of rank $\mu$ over $\C\llcurlybracket t\rrcurlybracket$ and over $\C\llcurlybracket \partial_t^{-1}\rrcurlybracket$, where the action of $t$ is given by
$t\cdot [\omega]=[f\omega]$ and the action of $\partial_t^{-1}$ is given by $\partial_t^{-1}\cdot [\omega]=[df\wedge\eta]$, where $d\eta=\omega$. 
The \emph{Gauss-Manin system} of $f$ at $P$ is
$$G_{f,P}:=H''_{f,P}[\partial_t].$$
Note that we have an injective map $H_{f,P}''\hookrightarrow G_{f,P}$ and a surjective map
$$H_{f,P}''\to \Omega^n_{f,P}:=\Omega_{X,P}^n/df\wedge\Omega_{X,P}^{n-1}.$$
A choice of local coordinates $x_1,\ldots,x_n$ at $P$ gives an isomorphism
$$\Omega^n_{f,P}\simeq\cO_{X,P}/J_{f,P}.$$

On the Gauss-Manin system there is a $V$-filtration, similar to the one discussed in Section~2, such that $\partial_tt-\beta$ is nilpotent on
${\rm Gr}_V^{\beta}G_{f,P}$ for all $\beta\in\Q$. This $V$-filtration induces a $V$-filtration on the submodule $H''_{f,P}$ and then a quotient
filtration on $\Omega^n_{f,P}$. An important fact is that for every $\beta\in\Q$, we have an isomorphism
\begin{equation}\label{eq2_spectrum}
{\rm Gr}^{\beta}_V\Omega_{f,P}^n\simeq {\rm Gr}_F^qH^{n-1}(F_{f,P},\C)_{\lambda},
\end{equation}
where $\lambda={\rm exp}(-2\pi i\beta)$ and $q=n-\lceil\beta\rceil$
(see \cite[(1.2.2)]{Saito_et_al}). 

The key result for us is that via the isomorphism $\Omega_{f,P}^n\simeq\cO_{X,P}/J_{f,P}$, the 
quotient $V$-filtration on $\Omega_{f,P}^n$ coincides with the quotient filtration on $\cO_{X,P}/J_{f,P}$
induced by the microlocal $V$-filtration: for every $\beta\in \Q$, the isomorphism identifies
\begin{equation}\label{eq3_spectrum}
V^{\beta}\Omega_{f,P}^n\quad \text{with}\quad (\widetilde{V}^{\beta}\cO_{X,P}+J_{f,P})/J_{f,P}
\end{equation}
(see \cite[Proposition~1.4]{Saito_et_al}). 

Recall now that if $\beta'\leq \widetilde{\alpha}_P(f)+1$, then $J_f\subseteq \widetilde{V}^{\beta'}\cO_X$ around $P$ (see 
Proposition~\ref{eq_Hodge_microlocal}). We thus deduce from (\ref{eq2_spectrum}) and (\ref{eq3_spectrum})
that if $\beta<\widetilde{\alpha}_P(f)+1$, then we have an isomorphism
$$\widetilde{V}^{\beta}\cO_{X,P}/\widetilde{V}^{>\beta}\cO_{X,P}\simeq {\rm Gr}_F^qH^{n-1}(F_{f,P},\C)_{\lambda},$$
where $\lambda={\rm exp}(-2\pi i\beta)$ and $q=n-\lceil\beta\rceil$. We thus obtain the following:

\begin{prop}
If $f$ has an isolated singularity at $P$ and $\beta<\widetilde{\alpha}_P(f)+1$, then $\beta$ is a jumping number for the
microlocal multiplier ideals of $f$ at $P$ (that is, $\widetilde{V}^{\beta}\cO_{X,P}\neq \widetilde{V}^{>\beta}\cO_{X,P}$)
if and only if $\beta$ is in the spectrum of $f$ at $P$. More precisely, the multiplicity of $\beta$ in this spectrum is equal to
$\dim_{\C}(\widetilde{V}^{\beta}\cO_{X,P}/\widetilde{V}^{>\beta}\cO_{X,P})$.
\end{prop}

We also obtain the result stated in the Introduction.

\begin{proof}[Proof of Theorem~\ref{thm4}]
If $\mu_{s(t)}(f_t)=0$ (that is, if the hypersurface defined by $f_t$ is smooth at $s(t)$), then the assertions in the theorem are trivial,
since all Hodge ideals coincide with the corresponding structure sheaves. Hence from now on we assume that $\mu_{s(t)}(f_t)>0$.
By Varchenko's theorem \cite{Varchenko2}, the constancy of the Milnor number implies that the spectrum of $f_t$ is independent of $t$. 
Since by assumption $p+\gamma\leq\alpha+1$, it follows from Proposition~\ref{eq_Hodge_microlocal} that
$$I_p(f_t^{\gamma})=\widetilde{V}^{p+\gamma}\cO_{{\mathcal X}_t}(f_t)\supseteq J_{f_t}$$
(note that it is enough to check these at $s(t)$, since the hypersurface defined by $f$ is smooth away from this point).
By the definition of the spectrum (see equation (\ref{eq1_spectrum})) and using equations (\ref{eq2_spectrum}) and (\ref{eq3_spectrum}), we conclude that
\begin{equation}\label{eq4_spectrum}
\dim_{\C}\big(\cO_{{\mathcal X}_t}/I_p(f_t^{\gamma})\big)\quad\text{is independent of}\quad t\in T.
\end{equation}

Since each hypersurface defined by every $f_t$ is reduced, it follows that the hypersurface defined by $f$ is reduced as well.
Let ${\mathcal Z}$ be the closed subscheme
of ${\mathcal X}$ defined by $I_p(f^{\gamma})$ and $\tau\colon {\mathcal Z}\to T$ the morphism induced by $\varphi$.  
Note that $\tau$ is finite, being proper, with finite fibers: in fact, the support of ${\mathcal Z}$ is a closed subset of $s(T)$.

Using the
Restriction Theorem for Hodge ideals (see \cite[Theorem~A(vi)]{MP1}), we see that
\begin{equation}\label{eq_conseq_restriction2}
I_p(f_t^{\gamma})\subseteq I_p(f^{\gamma})\cdot\cO_{{\mathcal X}_t}  \quad\text{for all}\quad t\in T,
\end{equation}
with equality for general $t\in T$. 
 For the finite morphism $\tau$, we know that the map
$$T\ni t\mapsto \dim_{\C}\big(\cO_{{\mathcal X}_t}/I_p(f^{\gamma})\cdot \cO_{{\mathcal X}_t}\big)$$
is upper-semicontinuous, and it is constant if and only if $\tau$ is flat.
We thus deduce from (\ref{eq4_spectrum})  and (\ref{eq_conseq_restriction2}) 
that the inclusion in (\ref{eq_conseq_restriction2})
is an equality for all $t\in T$ and that $\tau$ is flat.
This completes the proof of the theorem.
\end{proof}

\section{Proof of Teissier's conjecture}

Before giving the proof of Theorem~\ref{thm_main}, we make some preliminary remarks. We begin by reviewing 
some facts about semicontinuity of minimal exponents, that we will also use in the next section. 

Suppose first that we have a smooth morphism $\varphi\colon {\mathcal X}\to T$ of complex algebraic varieties and $s\colon T\to {\mathcal X}$
is a section of $\varphi$. Suppose that $f\in\cO_{\mathcal X}({\mathcal X})$ is such that for every $t\in T$, the restriction $f_t$ of $f$ to 
${\mathcal X}_t=\varphi^{-1}(t)$ is nonzero. We assume that $f\circ s=0$, hence we may consider $\widetilde{\alpha}_{s(t)}(f_t)$ for all $t\in T$.
In this case, the function 
$$T\ni t\to \widetilde{\alpha}_{s(t)}(f_t)$$
is lower semicontinuous by \cite[Theorem~E(2)]{MP2} (when each $f_t$ has an isolated singularity at $s(t)$, this result was also proved in 
\cite[Theorem~2.11]{Steenbrink}). In fact, the proof in \cite{MP2} shows something stronger: for every $\alpha>0$, the set 
$\{t\in T\mid\widetilde{\alpha}_{s(t)}(f_t)\geq\alpha\}$ is open in $T$. Since a countable intersection of nonempty open subsets is nonempty,
it follows that the set $\{\widetilde{\alpha}_{s(t)}(f_t)\mid t\in T\}$ has a maximum, which is achieved on an open subset of $T$. 

We next make two remarks concerning the hypothesis in Theorem~\ref{thm_main}.

\begin{rmk}\label{rmk_isolated_sing}
In the statement of Theorem~\ref{thm_main}, we may assume also that the hypersurface defined by $f\vert_H$ in $H$ 
has an isolated singularity at $P$.
For this,
it is enough to show that there is a smooth hypersurface $H'$ containing $P$, with the hypersurface defined by
$f\vert_{H'}$ having an isolated singularity at $P$, and
 such that $\widetilde{\alpha}_P(f\vert_{H'})\geq
\widetilde{\alpha}_P(f\vert_H)$. It is clear 
that in the statement of the theorem we may assume that
$X$ is affine and that we have a system of algebraic coordinates $x_1,\ldots,x_n$ on $X$, centered at $P$, such that $H$ is generated by $x_1$.
If $H'$ is defined by $a_1x_1+\ldots+a_nx_n$, where $a_1,\ldots,a_n\in\C$ are general, then it follows from the semicontinuity of minimal exponents
that $\widetilde{\alpha}_P(f\vert_{H'})\geq
\widetilde{\alpha}_P(f\vert_H)$. On the other hand, since the hypersurface $Y$ defined by $f$ has an isolated singularity at $P$,
it is easy to see that also the hypersurface in $H'$ defined by $f\vert_{H'}$ has an isolated 
singularity at $P$. This proves our assertion.
\end{rmk}

\begin{rmk}\label{rmk_case_polynomial}
In order to prove Theorem~\ref{thm_main}, it is enough to consider the case when $X=\A^n$, $P=0$, and $H$ is the hyperplane defined by $x_n=0$.
Indeed, we may assume that $X$ is affine and we have a system of algebraic coordinates $x_1,\ldots,x_n$ centered at $P$ such that $H$ is defined by $x_n$.
In this case, the map $\sigma=(x_1,\ldots,x_n)\colon X\to\A^n$ is \'{e}tale, $\sigma(P)=0$, and $H$ is the inverse image of the hyperplane $H'$, defined
by the vanishing of the last coordinate. By Remark~\ref{rmk_isolated_sing}, we may also assume that $f\vert_H$ has isolated singularities. 
For every $d\geq 1$, there is $f_d\in\cO_{\A^n}(\A^n)$ such that $f-\sigma^*(f_d)\in (x_1,\ldots,x_n)^d$. Since $f$ has an isolated singularity at $P$, it follows that 
$f$ and $\sigma^*(f_d)$ differ by an analytic change of coordinates for $d\gg 0$ (see \cite[Corollary~2.24]{GLS}); since both the minimal exponent and Teissier's 
invariant $\theta_P(f)$  can be computed
by passing to the local ring of the corresponding complex manifold, we have
$$\widetilde{\alpha}_P(f)=\widetilde{\alpha}_P\big(\sigma^*(f_d)\big)\quad\text{and}\quad \theta_P(f)=\theta_P\big(\sigma^*(f_d)\big).$$
The same argument implies that $\widetilde{\alpha}_P(f\vert_H)=\widetilde{\alpha}_P\big(\sigma^*(f_d)\vert_{H}\big)$ for $d\gg 0$.
On the other hand, since $\sigma$ induces a biholomorphic map in a neighborhood of $P$, we also have
$$\widetilde{\alpha}_P\big(\sigma^*(f_d)\big)=\widetilde{\alpha}_0(f_d),\quad\widetilde{\alpha}_P\big(\sigma^*(f_d)\vert_H\big)=\widetilde{\alpha}_0(f_d\vert_{H'}),
\quad\text{and}\quad \theta_P\big(\sigma^*(f_d)\big)=\theta_0(f_d).$$
This completes the proof of our assertion.
\end{rmk}

We can now give the proof of our main result.

\begin{proof}[Proof of Theorem~\ref{thm_main}]
It follows from Remarks~\ref{rmk_isolated_sing} and \ref{rmk_case_polynomial} that we may and will assume that $X=\A^n$, with coordinates
$x_1,\ldots,x_n$, $H$ is the hyperplane defined by $x_n=0$, and $P$ is the origin; moreover, $g=f(x_1,\ldots,x_{n-1},0)$
has an isolated singularity at $0$ in ${\mathbf A}^{n-1}$. 
We choose a positive integer $d$ such that $m:=d\big(\theta_P(f)+1\big)$ is an integer.
The case when the hypersurface defined by $f$ is smooth at $P$ is trivial, since in this case $\widetilde{\alpha}_P(f)=\infty$. Therefore from 
now on we assume that this hypersurface is singular at $P$. We thus have $J_f\subseteq (x_1,\ldots,x_n)$, hence $\theta_P(f)\geq 1$. 
Let $\lambda=\widetilde{\alpha}_P(g)+\frac{1}{\theta_P(f)+1}$, so that we need to show that $\widetilde{\alpha}_P(f)\geq\lambda$. 

If $n=2$, then $\widetilde{\alpha}_P(f)={\rm lct}_P(f)$ since by \cite[Theorem (0.4)]{Saito_microlocal}, we have $\widetilde{\alpha}_P(f) \leq n/2 = 1$. Therefore
 the assertion in the theorem follows from the main result in \cite{EM}. Hence from now on we 
assume $n\geq 3$. In this case, since $g$ has an isolated singularity at $0$, it follows that it is reduced in a neighborhood of $0$.

As we have explained in the Introduction, we follow the approach in \cite{EM}, replacing log canonical thresholds and multiplier ideals by minimal exponents
and Hodge ideals, respectively. For technical reasons, in our setting it is important to consider the following two-parameter family of hypersurfaces:
let 
$$h(x_1,\ldots,x_n,y,z)=f(x_1,\ldots,x_{n-1},yx_n^d)+zx_n^m\in\C[x_1,\ldots,x_n,y,z]$$
and for every $s,t\in \C$ we consider $h_{s,t}=h\vert_{y=s,z=t}\in\C[x_1,\ldots,x_n]$. 
Note that for every $s$ and $t$ we have $h_{s,t}\vert_H=f\vert_H$. First, since $f\vert_H$ is reduced in a neighborhood of $P$,
it follows that every $h_{s,t}$ is reduced in a neighborhood of $P$ and $h$ is reduced in a neighborhood of $\{P\}\times\A^2$.
We can thus use the results in Section~2 for the Hodge ideals of $h_{s,t}$ and $h$ around the respective subsets.
Second, we have
\begin{equation}\label{eq1_thm_main}
\widetilde{\alpha}_P(h_{s,t})\geq\widetilde{\alpha}_P(g)\quad\text{and}\quad\widetilde{\alpha}_{(P,s,t)}(h)\geq \widetilde{\alpha}_P(g)
\quad\text{for all}\quad s,t\in\C
\end{equation}
since the minimal exponent does not go up under restriction to a smooth hypersurface, see \cite[Theorem~E(1)]{MP2}. 
Since $\lambda-\widetilde{\alpha}_P(g)\leq\frac{1}{2}$, we deduce from 
Proposition~\ref{eq_Hodge_microlocal} that if we write 
$$\lambda=p+\gamma,\quad\text{with}\quad p=\lceil\lambda\rceil-1,$$ then
$$
I_p(h_{s,t}^{\gamma})=\widetilde{V}^{\lambda}\cO_{\A^n}(h_{s,t})\quad\text{around}\,\,P\,\,\quad\text{for all}\quad s,t\in\C\quad\text{and}
$$
$$
I_p(h^{\gamma})=
\widetilde{V}^{\lambda}\cO_{\A^{n+2}}(h)\quad\text{around}\quad \{P\}\times\A^2.
$$

We put $I:=I_p(h^{\gamma})\subseteq S=\C[x_1,\ldots,x_n,y,z]$. We consider on $S$ the grading such that
${\rm wt}(x_n)=1$, ${\rm wt}(y)=-d$, ${\rm wt}(z)=-m$, and ${\rm wt}(x_i)=0$ for $1\leq i\leq n-1$, so that $h$ is homogeneous of weight $0$.
This implies that $I$ is a graded ideal of $S$ (since the hypersurface defined by $h$ is preserved by the ${\mathbf C}^*$-action corresponding
to the grading of $S$, the same holds for the Hodge ideals of this hypersurface).

\noindent {\bf Step 1}. We first show that $x_n^{d-1}\in
I_p(h\vert_{y=0}^{\gamma})$
around $\{P\}\times\{0\}\times\A^1$. 
Note that $h\vert_{y=0}=g(x_1,\ldots,x_{n-1})+zx_n^{m}$. We thus deduce from the Thom-Sebastiani theorem for microlocal multiplier ideals
(see \cite[Theorem~2.2]{MSS}) that around $\{P\}\times\{0\}\times \A^1$ we have
\begin{equation}\label{eq2_thm_main}
\widetilde{V}^{\lambda}\cO_{\A^{n+1}}(h\vert_{y=0})=
\sum_{\beta_1+\beta_2=\lambda}\big(\widetilde{V}^{\beta_1}\cO_{\A^{n-1}}(g)\cdot\C[x_1,\ldots,x_n,z]\big)\cdot
\big(\widetilde{V}^{\beta_2}
\cO_{\A^2}(zx_n^{m})\cdot\C[x_1,\ldots,x_n,z]\big).
\end{equation}
By the characterization of the minimal exponent in terms of microlocal multiplier ideals (see equation (\ref{eq1_equality_Hodge_microlocal})), 
we have $\widetilde{V}^{\beta_1}\cO_{\A^{n-1}}(g)
=\C[x_1,\ldots,x_{n-1}]$ around the origin for $\beta_1=\widetilde{\alpha}_0(g)$. On the other hand, if $\beta_2=\frac{1}{\theta_P(f)+1}$,
then $\beta_2<1$ and thus the corresponding microlocal multiplier ideal is a usual multiplier ideal for a simple normal crossing divisor,
which is easy to compute:
$$\widetilde{V}^{\beta_2}\cO_{\A^2}(zx_n^{m})=\cJ\big(\A^2, (zx_n^m)^{\beta_2-\epsilon}\big)\quad\text{for}\quad 0<\epsilon\ll 1$$
$$=\cJ\big(\A^2, z^{\beta_2-\epsilon}x_n^{m(\beta_2-\epsilon)}\big)=(x_n)^{\lfloor m(\beta_2-\epsilon)\rfloor}=(x_n^{d-1})$$
(recall that $m/\big(\theta_P(f)+1\big)=d$). We thus conclude from (\ref{eq2_thm_main}) that around 
$\{P\}\times\{0\}\times\A^1$,
$x_n^{d-1}$ lies in 
$\widetilde{V}^{\lambda}\cO_{\A^{n+1}}(h\vert_{y=0})=I_p(h\vert_{y=0}^{\gamma})$.

\noindent {\bf Step 2}. We next show that there is an open subset $U$ of $\A^2$ such that for every $(s,t)\in U$, we have
$x_n^{d-1}\in I_p(h_{s,t}^{\gamma})$ around $P$. 
Note first that we have 
$$I_p(h\vert_{y=0}^{\gamma})\subseteq I\cdot \cO_{\A^{n+2}}\vert_{y=0}\quad\text{around}\quad \{P\}\times \{0\}\times\A^1.$$
This is a consequence of the Restriction Theorem for Hodge ideals in \cite[Theorem~A(vi)]{MP1}. 
We deduce from Step 1 that $x_n^{d-1}\in I\cdot \cO_{\A^{n+2}}\vert_{y=0}$ around $\{P\}\times\{0\}\times\A^1$. 
This implies that there is a polynomial $q\in\C[x_1,\ldots,x_n,z]$ such that $q\cdot x_n^{d-1}\in I\cdot \cO_{\A^{n+2}}\vert_{y=0}$ 
and $1-q\in (x_1,\ldots,x_n)\C[x_1,\ldots,x_n,z]$. Indeed, for every $a\in\C$ we have a polynomial $q_a\in\C[x_1,\ldots,x_n,z]$ such that $q_a\cdot x_n^{d-1}\in
I\cdot \cO_{\A^{n+2}}\vert_{y=0}$ and $q_a(0,a)\neq 0$. If we write $q_a=q'_a+q''_a$, with 
$$q'_a\in (x_1,\ldots,x_n)\C[x_1,\ldots,x_n,z]\quad \text{and}\quad q''_a\in\C[z],$$
we see that the gcd of the $q''_a$ is $1$, hence they generate the unit ideal. This immediately implies the existence of $q$ as asserted.

We can thus find polynomials $Q_1,Q_2\in S$, with $Q_1\in I$, such that
\begin{equation}\label{eq11_thm_main}
q\cdot x_n^{d-1}=Q_1+yQ_2.
\end{equation}
If we use the grading that we defined on $S$ and for any polynomial $Q\in S$, we denote by $Q_j$ the homogeneous component of $Q$
of weight $j$, then we can write
$$q_0=1+\sum_{i\geq 0}z^ix_n^{mi}u_i(x_1,\ldots,x_{n-1}),\quad\text{with}\quad u_0\in (x_1,\ldots,x_{n-1})\C[x_1,\ldots,x_{n-1}], \quad\text{and}$$
$$(yQ_2)_{d-1}=y\cdot \sum_{j,k\geq 0}z^jy^kx_n^{mj+d(k+2)-1}v_{j,k}(x_1,\ldots,x_{n-1}).$$
Since $I$ is graded, it follows from (\ref{eq11_thm_main}) that $q_0\cdot x_n^{d-1}-(yQ_2)_{d-1}\in I$. We thus see that if
$$R=1+\sum_{i\geq 0}z^ix_n^{mi}u_i(x_1,\ldots,x_{n-1})-y\cdot\sum_{j,k\geq 0}z^jy^kx_n^{d(k+1)+mj}v_{j,k}(x_1,\ldots,x_{n-1}),$$
then $R(0,y,z)=1$ and $R\cdot x_n^{d-1}\in I$. 

The second half of the Restriction Theorem for Hodge ideals (see \cite[Theorem~A(vi)]{MP1}) says that there is an open subset $U$ of $\A^2$
such that for every $(s,t)\in U$, we have
$$I_p(h_{s,t}^{\gamma})=I\cdot\cO_{\A^{n+2}}\vert_{y=s,z=t}.$$
Since $R(P,s,t)\neq 0$, we conclude that for all such $(s,t)$, we have $x_n^{d-1}\in I_p(h_{s,t}^{\gamma})$ around $P$.

\noindent {\bf Step 3}. We now prove that $x_n^{d-1}\in I_p(h_{1,0}^{\gamma})$ around $P$. Note first that
 the condition that $m\geq d\big(\theta_P(f)+1\big)$ implies that there is an open neighborhood $V$ of 
$(1,0)\in\A^2$ and an integer $\mu$ such that for every $(s,t)\in V$, the hypersurface $h_{s,t}$ has an isolated singularity at $P$,
with Milnor number $\mu$. Indeed, for every $(s,t)\in\A^2$, with $s\neq 0$, an easy change of variable 
implies that $h_{s^d,t}$ has the same Milnor number at $P$ as
\begin{equation}\label{eq13_main_thm}
f(x_1,\ldots,x_{n-1},x_n^d)+\frac{t}{s^m}x_n^m.
\end{equation}
It follows from the proof of \cite[Lemma~2.10]{EM} 
that the inequality $m\geq d\big(\theta_P(f)+1\big)$ implies the existence of an open neighborhood $V_0$ of $0$ in $\A^1$ such that
the Milnor number at $P$ of the hypersurface in (\ref{eq13_main_thm}) is finite and independent of $t/s^m\in V_0$. If $V$ is the image of the open set
$\{(s,t)\mid s\neq 0,t/s^m\in V_0\}$ via the map ${\mathbf A}^2\to {\mathbf A}^2$ that maps $(s,t)$ to $(s^d,t)$,
then $V$ satisfies the required property (note that this map is open, being flat).

It is then easy to see that there is an open neighborhood $W$ of $\{P\}\times V$ in $\A^n\times V$ such that for every $(s,t)\in V$, the only singular
point of $h_{s,t}$ in $W\cap \big(\A^n\times\{(s,t)\}\big)$ is $(P,s,t)$ (see for example \cite[Proposition~2.9(ii)]{EM}). 
Let $Z$ be the closed subscheme
of $W$ defined by $I\cdot\cO_W$ and $\tau\colon Z\to V$ the morphism induced by the projection $\A^n\times\A^2\to\A^2$.  
It follows from our choice of $W$ that we can apply 
Theorem~\ref{thm4} to the morphism $\tau$ and to the function $h\vert_W$
to conclude 
that $\tau$ is a finite flat morphism and we have
\begin{equation}\label{eq_conseq_restriction}
I_p(h_{s,t}^{\gamma})=I\cdot\cO_W\vert_{y=s,z=t}\quad\text{for all}\quad (s,t)\in V.
\end{equation}

Since $V$ is an irreducible variety and $\tau$ is finite and flat, the fact that $x_n^{d-1}\in I_p(h_{s,t}^{\gamma})$ for all $(s,t)\in U\cap V$ gives $x_n^{d-1}\in I\cdot\cO_W$ on $W$. Indeed, 
recall first that by Step 2, for every $(s,t)\in U\cap V$, we have $x_n^{d-1}\in I\cdot\cO_W\vert_{y=s,z=t}$ (a priori, we only know this around $P$, but $P$ is the only singular point of
$h_{s,t}$ in $W\cap \big(\A^n\times\{(s,t)\}\big)$). Since 
$\tau$ is finite and flat, this implies\footnote{Note that if $A$ is a reduced algebra of finite type over ${\mathbf C}$ and $B$ is a finite flat $A$-algebra,
then the injective homomorphism $A\hookrightarrow \prod_{\frm\in {\rm Max}(A)}A/\frm$ induces an injective homomorphism
$B\hookrightarrow\prod_{\frm\in {\rm Max}(A)}B/\frm B$.} that
$x_n^{d-1}\vert_{\tau^{-1}(U\cap V)}=0$. 
Using the fact that $V$ is a reduced scheme and $\tau$ is flat, we now deduce that $x_n^{d-1}\vert_Z=0$.
In particular, the equality in (\ref{eq_conseq_restriction}) for $(s,t)=(1,0)$ gives $x_n^{d-1}\in I_p(h_{1,0}^{\gamma})$ around $P$.

We can now conclude. Note that
$h_{1,0}(x_1,\ldots,x_n)=f(x_1,\ldots,x_{n-1},x_n^d)$. Consider the finite morphism 
$\varphi\colon \A^n\to\A^n$ given by $\varphi(x_1,\ldots,x_n)=(x_1,\ldots,x_{n-1},x_n^d)$. 
Since $\varphi^*(f)=h_{1,0}$ and 
$P$ is the only point in the fiber of $\varphi$ over $P$, it follows from Step 3 that we may
apply Theorem~\ref{thm2} for the restriction of $\varphi$ over a suitable neighborhood of $P$ 
to conclude that $1\in I_p(f^{\gamma})$ around $P$.
We thus have $\widetilde{\alpha}_P(f)\geq p+\gamma=\lambda$ by the characterization of the minimal exponent in terms of Hodge ideals
(see equation (\ref{eq2_equality_Hodge_microlocal})).
This completes the proof of the theorem.
\end{proof}

\begin{proof}[Proof of Corollary~\ref{cor_thm_main}]
Since $H_1,\ldots,H_{n-1}$ are general smooth
hypersurfaces in $X$ containing $P$, it follows that each $Z_i:=H_1\cap\ldots\cap H_i$ is smooth and $f\vert_{Z_i}$ has an isolated singularity at $P$, for 
$1\leq i\leq n-1$. We can thus apply Theorem~\ref{thm_main} to each of $f, f\vert_{Z_1},\ldots, f\vert_{Z_{n-2}}$ to conclude that
$$\widetilde{\alpha}_P(f)\geq\frac{1}{\theta_P(f)+1}+\frac{1}{\theta_P(f\vert_{Z_1})+1}+\ldots+\frac{1}{\theta_P(f\vert_{Z_{n-2}})+1}
+\widetilde{\alpha}_P(f\vert_{Z_{n-1}}).$$
Note now that $Z_{n-1}$ is a smooth curve. If $m={\rm mult}_P(f\vert_{Z_{n-1}})$, then $\widetilde{\alpha}(f\vert_{Z_{n-1}})=\frac{1}{m}=
\frac{1}{\theta_P(f\vert_{Z_{n-1}})+1}$. We thus obtain the inequality in the statement of the corollary.
\end{proof}

\section{A lower bound for the minimal exponent of general hyperplane sections}

Our goal in this section is to prove Theorem~\ref{thm5}. We begin with a few comments regarding what we mean by 
restriction to a \emph{general hypersurface}. Let $P$ be a point on a smooth complex algebraic variety $X$. We consider 
a system of regular functions
$y_1,\ldots,y_N$ defined on an open neighborhood of $P$ and whose images in the local ring $\cO_{X,P}$ generate the maximal ideal. 
If $H$ is a hypersurface in $X$ defined around $P$ by a linear combination $\sum_{i=1}^Na_iy_i$, with $a_1,\ldots,a_N\in\C$ general,
we refer to $H$ as a general hypersurface in $X$ containing $P$. It is clear that such a hypersurface is smooth at $P$.

If $f\in\cO_X(X)$ is such that $f(P)=0$, then it follows from the semicontinuity statement for minimal exponents discussed at the beginning of the previous section
that for such a general $H$, the minimal exponent $\widetilde{\alpha}_P(f\vert_H)$ is independent of $H$. Moreover, this is the largest of all minimal exponents
$\widetilde{\alpha}_P(f\vert_{H'})$, where $H'$ is \emph{any} hypersurface in $X$ that is smooth at $P$, and with $f\vert_{H'}\neq 0$ (this follows by enlarging the given system of generators
of the maximal ideal of $\cO_{X,P}$ with the germ of an equation defining $H'$ around $P$). 

We can now prove the lower bound for the minimal exponent for the restriction to a general hypersurface containing $P$.

\begin{proof}[Proof of Theorem~\ref{thm5}]
Let $d={\rm mult}_P(f)$. The assertion in the theorem is clear if the hypersurface defined by $f$ is smooth at $P$:
indeed, in this case also the hypersurface defined by $f\vert_H$ in $H$ is smooth, hence $\widetilde{\alpha}_P(f\vert_H)=\infty
=\widetilde{\alpha}_P(f)$. Therefore from now on we may assume $d\geq 2$.

Suppose first that we know the assertion when $f$ has an isolated singularity at $P$. In order to deduce the general case,
we may and will assume that $X$ is affine and we have an algebraic system of coordinates $x_1,\ldots,x_n$ on $X$,
centered at $P$. It follows from the discussion at the beginning of this section that it is enough to exhibit one smooth hypersurface
$H$ containing $P$ such that $\widetilde{\alpha}_P(f\vert_H)\geq\widetilde{\alpha}_P(f)-\frac{1}{d}$. 
For every $m>d$, consider
$$f_m=f+\sum_{i=1}^na_{i,m}x_i^m,$$
where $a_{1,m},\ldots,a_{n,m}\in\C$ are general. Note that by the Kleiman-Bertini theorem
the hypersurface defined by $f_m$ has at most one singular point, namely $P$. 
Note also that ${\rm mult}_P(f_m)=d$.
By the isolated singularity case, it follows that if
$H$ is general (depending on $f_m$), then
$$\widetilde{\alpha}_P(f_m\vert_H)\geq \widetilde{\alpha}_P(f_m)-\frac{1}{d}.$$
Since the intersection of a countable family of nonempty Zariski open subsets of an irreducible complex algebraic variety is nonempty,
it follows that we can choose such a smooth hypersurface $H$ that satisfies these conditions for all $m$.
Since we have
$$\lim_{m\to\infty}\widetilde{\alpha}_P(f_m)=\widetilde{\alpha}_P(f)\quad\text{and}\quad \lim_{m\to\infty}\widetilde{\alpha}_P(f_m\vert_H)
=\widetilde{\alpha}_P(f\vert_H)$$
by \cite[Proposition~6.6(3)]{MP2}, we deduce that
$$\widetilde{\alpha}_P(f\vert_H)\geq \widetilde{\alpha}_P(f)-\frac{1}{d}.$$
Therefore $f$ satisfies the assertion in the theorem.

We thus see that it is enough to treat the case when $f$ has an isolated singularity at $P$. This follows from a sharper inequality
proved by Loeser in \cite{Loeser}. For the benefit of the reader, we include a slightly modified version of his proof, explaining in detail
how various arguments from \cite{Teissier1} come in the picture. 

Arguing as in Remark~\ref{rmk_case_polynomial}, we see that it is enough to consider the case when $X=\A^n$ and $P=0$. After a suitable choice of
coordinates $x_1,\ldots,x_n$, we may assume that $H$ is the hyperplane defined by $x_n=0$. Consider the following family
of polynomials 
$$h_t(x_1,\ldots,x_n)=f(x_1,\ldots,x_{n-1},tx_n)+(1-t)x_n^d,\quad\text{for}\quad t\in\C.$$
Note that $h_0=g+x_n^d$, where $g=f(x_1,\ldots,x_{n-1},0)$, hence the Thom-Sebastiani formula for the minimal exponent gives
$$\widetilde{\alpha}_0(h_0)=\widetilde{\alpha}_0(g)+\frac{1}{d}$$
(see \cite[Example~(6.8)]{Malgrange} or \cite[Example~6.7]{MP2}).
On the other hand, we have $h_1=f$, hence by the semicontinuity of the minimal exponent (see the discussion at the beginning of the previous section)
there is a Zariski open neighborhood $U$ of $1$ such that $\widetilde{\alpha}_0(h_t)\geq \widetilde{\alpha}_0(f)$ for all $t\in U$. The key point is to show that
there is an open neighborhood $V$ of $0$ such that for $t\in V$, $h_t$ has an isolated singularity at $0$ and the Milnor number is constant.
Indeed, in this case Varchenko's theorem \cite{Varchenko2} implies that for $t\in V$, we have 
$$\widetilde{\alpha}_0(h_t)=\widetilde{\alpha}_0(h_0)
=\widetilde{\alpha}_0(g)+\frac{1}{d}.$$
By taking $t\in U\cap V$, we thus obtain 
$$\widetilde{\alpha}_0(f\vert_H)=\widetilde{\alpha}_0(g)\geq\widetilde{\alpha}_0(f)-\frac{1}{d}.$$

Note that the Milnor number of $h_0$ at $0$ is $\mu_0(g+x_n^{d-1})=(d-1)\cdot \mu_0(g)$. 
Suppose now that $t\neq 0$ and we want to compute the Milnor number $\mu_t$ of $h_t$ at $0$ for general such $t$.
Note that since $h_0$ has an isolated singularity at $0$, the same holds for $h_t$, with $t$ general.

An easy change of variable gives
$$\mu_t=\dim_{\C}(\cO_{\A^n,0}/J_t),\quad\text{where}\quad
J_t=(\partial f/\partial x_1,\ldots,\partial f/\partial x_{n-1},dsx_n^{d-1}+\partial f/\partial x_n)\cdot\cO_{\A^n,0}$$
with $s=\frac{1-t}{t^d}$. We will make use of various facts about Hilbert-Samuel multiplicities, for which we refer to 
\cite[Chapter~14]{Matsumura}) and, more generally, of mixed multiplicities, for which we refer to \cite{Teissier1} or
\cite{Swanson}. Note that since $h_t$ has an isolated singularity at $0$, the elements
$\frac{\partial f}{\partial x_1},\ldots,\frac{\partial f}{\partial x_{n-1}}, dsx_n^{d-1}+\frac{\partial f}{\partial x_{n}}$ form a system of parameters,
hence a regular sequence, in the Cohen-Macaulay ring $\cO_{\A^n,0}$. It follows that if $\Gamma$ is defined by 
$\frac{\partial f}{\partial x_1},\ldots,\frac{\partial f}{\partial x_{n-1}}$, then its local ring $\cO_{\Gamma,0}$ is Cohen-Macaulay and
\begin{equation}\label{eq1_thm5}
\mu_t=e(J_t;\cO_{\A^n,0})=e\big((dsx_n^{d-1}+\partial f/\partial x_n)\cdot\cO_{\Gamma,0};\cO_{\Gamma,0}\big)
\end{equation}
(see \cite[Theorem~14.11]{Matsumura}). 
On the other hand, since $t$ is general, $s$ is general too, hence
\begin{equation}\label{eq2_thm5}
e\big((sdx_n^{d-1}+\partial f/\partial x_n)\cdot\cO_{\Gamma,0};\cO_{\Gamma,0}\big)=e\big((x_n^{d-1},\partial f/\partial x_n)\cdot\cO_{\Gamma,0};
\cO_{\Gamma,0}\big)
\end{equation}
(see \cite[Theorems~14.13 and 14.14]{Matsumura}). 
Using again the fact that $\frac{\partial f}{\partial x_1},\ldots,\frac{\partial f}{\partial x_{n-1}}$ form a regular sequence and
the definition of the Milnor number, we see that
\begin{equation}\label{eq3_thm5}
\mu_0(g)=\dim_{\C}\big(\cO_{\A^n,0}/(x_n,\partial f/\partial x_1,\ldots,\partial f/\partial x_{n-1})\big)=e\big((x_n)\cdot\cO_{\Gamma,0}; \cO_{\Gamma,0}\big).
\end{equation}

The fact that $H$ is general is used in two ways. First, the Jacobian $J(f\vert_H)$ of $f\vert_H$ and the restriction of $J_f$ to $\cO_{H,0}$
have the same integral closure (see \cite[Proposition~2.7]{Teissier1}). This implies that
\begin{equation}\label{eq4_thm5}
\mu_0(f\vert_H)=e(J_{f\vert_H};\cO_{H,0})=e(J_f\cdot \cO_{H,0};\cO_{H,0}\big).
\end{equation}
On the other hand, since $H$ is general, a basic property of mixed multiplicities (see \cite[Corollary~2.2]{Teissier1} or \cite[Theorem~2.5]{Swanson})
gives
\begin{equation}\label{eq5_thm5}
e(J_f\cdot\cO_{H,0};\cO_{H,0})=e(J_f^{[n-1]},\frm^{[1]};\cO_{\A^n,0})=e(\frm\cdot \cO_{\Gamma,0}; \cO_{\Gamma,0}),
\end{equation}
where $\frm$ is the maximal ideal in $\cO_{\A^n,0}$ (for the last equality, note that since $H$ is general, $\partial f/\partial x_1,\ldots,\partial f/\partial x_{n-1}$
are general linear combinations of a system of generators of $J_f$). 
By combining (\ref{eq4_thm5}) and (\ref{eq5_thm5}), we conclude that 
$\mu_0(g)=e(\frm\cdot \cO_{\Gamma,0}; \cO_{\Gamma,0})$ and using also (\ref{eq3_thm5}), it follows that
$$e\big((x_n)\cdot\cO_{\Gamma,0}; \cO_{\Gamma,0}\big)=e(\frm\cdot \cO_{\Gamma,0}; \cO_{\Gamma,0}).$$
The completion of $\cO_{\Gamma,0}$ has all its minimal primes of the same dimension (in fact, it is Cohen-Macaulay),
hence we can apply 
a theorem of Rees \cite{Rees} to conclude that the ideals $(x_n)\cdot\cO_{\Gamma,0}\subseteq \frm\cdot \cO_{\Gamma,0}$
have the same integral closure. Since $d={\rm mult}_0(f)$, it follows that
$\partial f/\partial x_n\in\frm^{d-1}$ and we see that the ideals
$(x_n^{d-1},\partial f/\partial x_n)\cdot\cO_{\Gamma,0}$ and $\frm^{d-1}\cdot\cO_{\Gamma,0}$ have the same integral closure. 
Therefore 
$$e\big((x_n^{d-1},\partial f/\partial x_n)\cdot\cO_{\Gamma,0};
\cO_{\Gamma,0}\big)=e\big(\frm^{d-1}\cdot\cO_{\Gamma,0};\cO_{\Gamma,0})=(d-1)\cdot e\big(\frm\cdot\cO_{\Gamma,0};\cO_{\Gamma,0}).$$
Using now (\ref{eq1_thm5}), (\ref{eq2_thm5}), and (\ref{eq3_thm5}), we conclude that for $t$ general, we have
$\mu_t=(d-1)\cdot \mu_0(g)=\mu_0(h_0)$. This completes the proof of the theorem.
\end{proof}

\begin{cor}\label{cor_thm5}
Let $X$ be a smooth complex algebraic variety of dimension $n$ and $P$ a point in $X$. Let $f\in\cO_X(X)$ be nonzero such that $f(P)=0$
and let $d={\rm mult}_P(f)$.
If $\widetilde{\alpha}_P(f)>1+\frac{r}{d}$, for some $r\leq n-1$ and 
$H_1,\ldots,H_r$ are general hypersurfaces in $X$ containing $P$, then
the hypersurface of $Y=H_1\cap\ldots\cap H_r$ defined by $f\vert_Y$ has rational singularities at $P$.
\end{cor}

\begin{proof}
A repeated application of Theorem~\ref{thm5} gives $\widetilde{\alpha}_P(f\vert_Y)>1$. This implies that the hypersurface in $Y$ defined
by $f\vert_Y$ has rational singularities at $P$ by \cite[Theorem~0.4]{Saito-B}.
\end{proof}

\begin{eg}
Let $n\geq 2$ and let $f={\rm det}(x_{i,j})_{1\leq i,j\leq n}$ be the determinant of an $n\times n$ matrix of indeterminates. 
In this case the reduced Bernstein-Sato polynomial of $f$ (at $0$) is given by 
$$\widetilde{b}_f(s)=\prod_{i=2}^n(s+i)$$
(see for example \cite[Appendix]{Kimura}). We thus have $\widetilde{\alpha}_0(f)=2$.
Since ${\rm mult}_0(f)=n$, it follows from Corollary~\ref{cor_thm5} that if $L\subseteq\A^{n^2}$ is a general
linear subspace containing $0$, of codimension $<n$, then the restriction $f\vert_L$ defines a hypersurface with rational 
singularities (note that $f\vert_L$ is a homogeneous polynomial, hence having rational singularities at $0$ implies rational singularities everywhere).
\end{eg}

\begin{cor}\label{cor_max_min_exp}
Let $X$ be a smooth complex algebraic variety of dimension $n$ and $P$ a point in $X$. If $f\in\cO_X(X)$ is nonzero and $P$ is a singular
point of the hypersurface $Y$ defined by $f$, then the following are equivalent:
\begin{enumerate}
\item[i)] We have $\widetilde{\alpha}_P(f)=\frac{n}{2}$.
\item[ii)] The tangent cone $C_P(f)$ of $Y$ at $P$ is a 
quadric cone of rank $n$.
\item[iii)] There are analytic local coordinates $x_1,\ldots,x_n$ on $X$ centered at $P$ such that $f=\sum_{i=1}^nx_i^2$. 

\end{enumerate}
\end{cor}

\begin{proof}
The equivalence between ii) and iii) is well-known: for example, it is a consequence of the Morse lemma. 
Note also that if $C_P(f)$ is a 
quadric cone of rank $n$, then it is well-known that $\widetilde{\alpha}_P(f)=\frac{n}{2}$ (see, for example, \cite[Theorem~E(3)]{MP2}).
Therefore it is enough to prove the converse.

Recall that since $P$ is a singular point of $Y$, we always have $\widetilde{\alpha}_P(f)\leq\frac{n}{2}$ by \cite[Theorem (0.4)]{Saito_microlocal}. 
We prove by induction
on $n\geq 1$ that if
$\widetilde{\alpha}_P(f)=\frac{n}{2}$, then $f$ satisfies the condition in ii). The case $n=1$ is trivial.
For the induction step, suppose that $n\geq 2$ and that we know the assertion for $n-1$.
If $\widetilde{\alpha}_P(f)=\frac{n}{2}$ and $d={\rm mult}_P(f)$, then for a general hypersurface
$H$ in $X$ containing $P$, it follows from Theorem~\ref{thm5} that
$$\widetilde{\alpha}_P(f\vert_H)\geq \widetilde{\alpha}_P(f)-\frac{1}{d}= \frac{n}{2}-\frac{1}{d}\geq \frac{n-1}{2}.$$
Since $\widetilde{\alpha}_P(f\vert_H)\leq\frac{n-1}{2}$, we conclude that $d=2$ and $\widetilde{\alpha}_P(f\vert_H)=\frac{n-1}{2}$.
By the induction hypothesis, it follows that $C_P(f\vert_H)$ is a quadric cone of rank $n-1$. In particular, we deduce that
$C_P(f)$ is a quadric cone of rank $\geq n-1$. By the Morse lemma, we conclude that there are local analytic 
coordinates $x_1,\ldots,x_n$ on $X$ centered at $P$ such that $f=x_1^m+\sum_{i=2}^nx_i^2$ for some $m\geq 2$. 
In this case the Thom-Sebastiani formula gives $\widetilde{\alpha}_P(f)=\frac{1}{m}+\frac{n-1}{2}$. Since 
$\widetilde{\alpha}_P(f)=\frac{n}{2}$, we deduce that $m=2$,
completing the proof of the induction step. 
\end{proof}

\end{document}